\documentclass[11pt,twoside,reqno]{amsart}
\usepackage{amsmath, amssymb, amsthm, color, tocvsec2}
\usepackage[colorlinks=false,citecolor=blue]{hyperref}
\usepackage[tmargin=1.2in,bmargin=1.2in,rmargin=1in,lmargin=1in]{geometry}

\newtheorem{theorem}{Theorem}[section]
\newtheorem{lemma}[theorem]{Lemma}
\newtheorem{prop}[theorem]{Proposition}
\newtheorem{cor}[theorem]{Corollary}
\theoremstyle{definition}
\newtheorem{definition}[theorem]{Definition}
\newtheorem{example}[theorem]{Example}
\theoremstyle{remark}
\newtheorem{remark}[theorem]{\bf{Remark}}
\newtheorem{question}[theorem]{\bf{Question}}
\numberwithin{equation}{section}
\allowdisplaybreaks

\newcommand{\blue}[1]{{\color{black}#1}}

\begin{document}

\title[Numerical radius and $\ell_p$ operator norm of Kronecker products
\blue{and Schur powers}]
{Numerical radius and $\ell_p$ operator norm of Kronecker products
\blue{and Schur powers}: inequalities and equalities}

\author{Pintu Bhunia}
\address[P.~Bhunia]{Department of Mathematics, SRM University AP,
Amaravati 522240, Andhra Pradesh, India \& Department of Mathematics,
Indian Institute of Science, Bangalore 560012, India}
\email{\tt pintubhunia5206@gmail.com; pintu.b@srmap.edu.in}

\author{Sujit Sakharam Damase}
\address[S.S.~Damase]{Department of Mathematics, Indian Institute of
Science, Bangalore 560012, India; PhD student}
\email{\tt sujits@iisc.ac.in}

\author{Apoorva Khare}
\address[A.~Khare]{Department of Mathematics, Indian Institute of
Science, Bangalore, India and Analysis \& Probability Research Group,
Bangalore, India}
\email{\tt khare@iisc.ac.in}

\subjclass[2020]{47A30 (primary); 46M05, 47A12, 15A60, 47A63, 26C10
(secondary)}

\keywords{Numerical radius, Kronecker product, Schur product,
\blue{Hadamard power}, $\ell_p$ operator norm, \blue{spectral
radius, spectral norm, maximum modulus eigenvalue, Jordan block}}

\date{\today}

\begin{abstract}
Suppose $A=[a_{ij}]\in \mathcal{M}_n(\mathbb{C})$ is a complex $n \times
n$ matrix and $B\in \mathcal{B}(\mathcal{H})$ is a bounded linear
operator on a complex Hilbert space $\mathcal{H}$.
We show that $w(A\otimes B)\leq w(C),$ where $w(\cdot)$
denotes the numerical radius and $C=[c_{ij}]$ with 
$c_{ij}= w\left(\begin{bmatrix}
         0& a_{ij}\\
         a_{ji}&0
    \end{bmatrix} \otimes B\right).$
This refines Holbrook's classical bound $w(A\otimes B)\leq w(A) \|B\|$
[\textit{J.\ Reine Angew.\ Math.}\ 1969], when all entries of $A$ are
non-negative.
If moreover $a_{ii}\neq 0$ $ \forall i$, we prove that
$w(A\otimes B)= w(A) \|B\|$ if and only if $w(B)=\|B\|.$
We then extend these and other results to the more general
setting of semi-Hilbertian spaces induced by a positive operator.

In the reverse direction, we also specialize these results
to Kronecker products and hence to Schur/entrywise products, of matrices:
\blue{(1)(a)~We first provide an alternate proof (using $w(A)$) of a
result of Goldberg--Zwas [\textit{Linear Algebra Appl.}\ 1974] that if
the spectral norm of $A$ equals its spectral radius, then each Jordan
block for each maximum-modulus eigenvalue must be $1 \times 1$ (``partial
diagonalizability'').
(b)~Using our approach, we further show given $m \geq 1$ that $w(A^{\circ
m})\leq w^m(A)$ -- we also characterize when equality holds here.}
(2)~We provide upper and lower bounds for the $\ell_p$ operator norm and
the numerical radius of $A\otimes B$ for all $A \in
\mathcal{M}_n(\mathbb{C})$, which become equal when restricted to doubly
stochastic matrices $A$.
Finally, using these bounds we obtain an improved estimation for the
roots of an arbitrary complex polynomial.
\end{abstract}
\maketitle

\tableofcontents
\settocdepth{section}

\section{Introduction and main results}\label{sec-int}

Throughout this work, $\mathcal{H}$ denotes an arbitrary but
fixed (nonzero) complex Hilbert space. The study of the numerical radius
of a bounded linear operator $B \in \mathcal{B}(\mathcal{H})$ goes back
at least to Toeplitz \cite{Toeplitz} and Hausdorff \cite{Hausdorff} (see
also \cite{Halmos}). In fact, it goes back even earlier to Rayleigh
quotients in the 19th century. In recent times, the numerical radius has
seen widespread usage through applications in functional analysis,
operator theory, numerical analysis, systems theory, quantum computing,
and quantum information theory. We refer the reader to e.g.~\cite{book}
for more on this.

In this work, we focus on the numerical radius of the Kronecker product
(tensor product) of two operators, a quantity that has also been well
studied (see e.g.\ \cite{Ando}, \cite{BhuniaPMSC},
\cite{Gau-IEOT}--\cite{Kittaneh}, \cite{Shiu}). The interested reader can
also see the norm of the derivative of the Kronecker products, studied by
Bhatia et al. \cite{Bhatia-ejla}.  We introduce the relevant notions
here.

\begin{definition}
The numerical range $W(B)$ of a bounded linear operator $B \in
\mathcal{B}(\mathcal{H})$ is defined as $W(B) := \left\{ \langle
Bx,x\rangle: x\in \mathcal{H}, \|x\|=1  \right\}$, and the associated
numerical radius $w(B)$ is defined as
$w(B) := \sup\left\{ |\langle Bx,x\rangle|: x\in \mathcal{H}, \|x\|=1
\right\} = \sup\left\{ |\lambda|: \lambda\in W(B)  \right\}.$
\end{definition}

It is well known that the numerical radius defines a norm on
$\mathcal{B}(\mathcal{H})$, which is equivalent to the operator norm
$\|B\|=\sup\left\{ \| Bx\|: x\in \mathcal{H}, \|x\|=1
\right\}$ via: $\frac{1}{2}\|B\| \leq w(B) \leq \|B\|$. It is also weakly
unitarily invariant (see e.g.\ \cite{Halmos-1}), i.e.\
$w(U^*BU) = w(B)$ for every unitary $U$.

\begin{definition}\label{D12}
The tensor product $\mathcal{K} \otimes \mathcal{H}$ of two complex
Hilbert spaces $\mathcal{K}, \mathcal{H}$ is defined as the completion of
the inner product space consisting of all elements of the form
$\sum_{i=1}^{n}x_i \otimes y_i$ for $x_i \in \mathcal{K}$ and $ y_i \in
\mathcal{H}$, for $n\geq 1$, under the  inner product $\langle x \otimes
y, z \otimes w \rangle :=\langle x , z\rangle \langle y , w\rangle.$ In
particular, $\mathbb{C}^n \otimes \mathcal{H} \cong \mathcal{H}^{\oplus
n}$, and we will denote this by $\mathcal{H}^n$ henceforth.

The Kronecker product $A \otimes B$ of two operators $A\in
\mathcal{B}(\mathcal{K})$ and $B\in \mathcal{B}(\mathcal{H})$ is defined
as $(A \otimes B)(x \otimes y) := Ax \otimes By$ for $x \otimes y \in
\mathcal{K} \otimes \mathcal{H}.$ In particular, if $A=[a_{ij}]\in
\mathcal{M}_n(\mathbb{C})$ and $B\in \mathcal{B}(\mathcal{H})$, the
Kronecker product  $A\otimes B :=  [a_{ij}B]_{i,j=1}^n \in \mathcal{B}(
\mathcal{H}^n)$ is an $n\times n$ operator matrix.
\end{definition}

With this notation in hand,
\blue{we begin by broadly describing our work. It develops numerical
radius bounds in three themes -- the first of which is classical, while
the others seem to be novel.
\begin{enumerate}
\item Inequalities and equalities for the numerical radius of $A\otimes
B$, where $A,B$ are operators over Hilbert and semi-Hilbert spaces --
with the motivating goal to improve on the 1969 upper bound by Holbrook
\cite[Theorem 3.4]{double}:
\begin{equation}\label{E0-1}
w(A \otimes B) \leq w(A) \| B \|.
\end{equation}
This is a question that has seen much subsequent activity in the
literature.

\item We initiate the study of numerical radius bounds for Schur/Hadamard
powers of complex matrices. To the best of our knowledge, these have not
been studied before.

\item Numerical radius and $\ell_p$-norm bounds for Kronecker products of
matrices. Surprisingly, the study of $\ell_p$-norm bounds seems to be
very recent, including joint work by one of us~\cite{Khare}. And bounds
for $\| A \otimes B \|_p$ have -- once again to our knowledge -- not been
studied earlier.
\end{enumerate}

\subsection{Main results 1: Refining Holbrook's bound}

We now present our main results in the three themes listed above, in
serial order. Holbrook proved his inequality~\eqref{E0-1}} in the setting
of bounded linear operators $A$ and $B$ on an arbitrary Hilbert space
$\mathcal{H}$; this easily generalizes to any $A \in
\mathcal{B}(\mathcal{K})$ and $B \in \mathcal{B}(\mathcal{H})$, where
$(\mathcal{K}, \mathcal{H})$ denotes an arbitrary pair of Hilbert spaces
-- see e.g.\ \cite[Equation (2)]{BhuniaPMSC}. Our goal is to refine this
inequality; we are able to achieve this when the Hilbert space
$\mathcal{K}$ is finite-dimensional. Here is our first main result.

\begin{theorem}\label{th4}
Let $A=[a_{ij}]\in \mathcal{M}_n(\mathbb{C})$ and $B\in
\mathcal{B}(\mathcal{H}).$ Then 
\begin{equation}
w(A\otimes B) \leq w(C) \leq  w(C^\circ),
\end{equation}
where $C=[c_{ij}], C^\circ = [c^\circ_{ij}]$ have diagonal entries
$c_{ii} = c^\circ_{ii} = |a_{ii}| w(B)$, and off-diagonal entries
\[ c_{ij}=   w\left(\begin{bmatrix}
0& a_{ij}\\
a_{ji}&0
\end{bmatrix} \otimes B\right), \quad
c^\circ_{ij} = |a_{ij}| \|B\|, \qquad \forall 1 \leq i \neq j \leq n.
\]
\end{theorem}

We make several remarks here:

\begin{enumerate}
\item Note that Theorem~\ref{th4} implies Holbrook's
inequality~\eqref{E0-1} in the special case where all $a_{ij} \geq 0$.
Indeed, using the Schur product one may rewrite Holbrook's inequality as:
$w(A \otimes B) \leq w(A \circ \| B \| {\bf 1}_{n \times n})$, where
${\bf 1}_{n \times n}$ is the all-ones matrix and $A \circ A'$ denotes
the Schur/entrywise product of two complex matrices $A,A'$. Our bound,
when all $a_{ij} \geq 0$, says that
\begin{equation}\label{ECtilde}
w(A \otimes B) \leq w(C^\circ) = w \left( A \circ \|
B \| {\bf 1}_{n \times n} - ( \| B \| - w(B)) A \circ I_n \right),
\end{equation}
and this is at most Holbrook's bound $w(A) \| B \|$, via
entrywise monotonicity of the numerical radius~\eqref{p2} below.
This refinement is moreover strict, as the following example shows. 
Let
$A=\begin{bmatrix}
    1&0\\
    0&2
\end{bmatrix}$. Then $w(A\otimes B)=  w\left(\begin{bmatrix}
 		1w(B)&  0  \|B\| \\
 		0 \|B\|& 2w(B) \\
 \end{bmatrix}
 	\right) =2w(B)<  2\|B\|= w(A)\|B\|$ if $w(B)< \|B\|.$

\item Going beyond matrices with non-negative real entries: if $A\in
\mathcal{M}_n(\mathbb{C})$ is normal (with possibly complex entries),
then Theorem \ref{th4} refines \eqref{E0-1} via the weakly unitarily
invariant property.

\item Theorem~\ref{th4} is a special case of an even stronger result,
in the setting of a semi-Hilbertian space $(\mathcal{H}, \langle \cdot,
\cdot \rangle_P)$ for any positive operator $P$. See Theorem~\ref{th4-4}.
\end{enumerate}

We now move to the question of when equality is attained in~\eqref{E0-1}.
Gau and Wu showed in~\cite{Gau2019}, a somewhat technical
characterization of when~\eqref{E0-1} is an equality for $A \in
\mathcal{M}_n(\mathbb{C})$ and $B \in \mathcal{M}_m(\mathbb{C})$. We
provide a different, simpler to state characterization.
Moreover, in it we assume $A$ has non-negative entries, but at the same
time allow $B$ to be much more general:

\begin{theorem} \label{cor1}
Let $A=[a_{ij}] \in \mathcal{M}_n(\mathbb{C})$ and $B\in
\mathcal{B}(\mathcal{H}).$ If
$w(B)=\|B\|$ then $w(A\otimes B)= w(A) \|B\|$.
Conversely, if all $a_{ij}\geq 0, a_{ii}\neq 0$, and
$w(A\otimes B)= w(A) \|B\|$, then $w(B)=\|B\|.$
\end{theorem}

\blue{
\subsection{Main results 2: Schur powers and partial diagonalizability}

We now turn to some applications of results bounding $w(A \otimes B)$.}
First, we study the numerical radius of Schur/entrywise powers of complex
matrices $w(A^{\circ m})$. In Proposition~\ref{th10}, we show that if
$A\in \mathcal{M}_n(\mathbb{C}),$ then $w(A^{\circ m}) \leq  w(A)
\|A\|^{m-1} \leq 2^{m-1} w(A)\ \forall m \geq 1$. When $w(A) = \|A\|$,
the first inequality becomes
$w(A^{\circ m}) \leq w^m(A)$; in Theorem~\ref{Tref}(2) we completely
characterize when this is an \blue{equality:

\begin{theorem}
Suppose $w(A) = \|A\|$ for $A \in \mathcal{M}_n(\mathbb{C})$, and $m \geq 1$.
Then $w(A^{\circ m}) = w^m(A)$ if and only if $A^{\otimes m}$ has an
eigenvector in the span of ${\bf e}_1^{\otimes m}, \ldots, {\bf
e}_n^{\otimes m}$ with eigenvalue $\|A\|^m$. Moreover, if this holds then
$w(A^{\circ m'}) = w(A)^{m'}$ for all $1 \leq m' \leq m$.
\end{theorem}

An interesting intermediate step here -- see Theorem~\ref{Tref}(1) -- is:

\begin{theorem}
Suppose $A$ is a complex square matrix whose spectral norm equals its
spectral radius. If $\lambda$ is any eigenvalue of $A$ of maximum
modulus, then every Jordan block for $\lambda$ is $1 \times 1$ (i.e., $A$
is ``partially diagonalizable'').
\end{theorem}

This leads to a natural speculative generalization of the Spectral
Theorem for normal operators; see Section~\ref{Sspeculation}.

\subsection{Main results 3: $\ell_p$-norm bounds}

Another application of the above results} (see Section~\ref{sec-lpnorm})
is to provide precise values for the $\ell_p$ operator norm and numerical
radius of (dilations of) doubly stochastic matrices.
Recall for $A\otimes B\in \mathcal{B}(\mathcal{H}^n)$ that
\begin{equation}\label{Eopnorm}
\| A\otimes B\|_p := \sup\left\{  \frac{ \| (A\otimes B) x \|_p}{\|x\|_p}
:  { x\in \mathcal{H}^n,\, x\neq 0} \right\}, \qquad 1\leq p \leq \infty
\end{equation}
is its $\ell_p$ operator norm, where $\|x\|_p=\left(
\sum_{j=1}^n\|x_j\|^p\right)^{1/p}$ for $x=(x_1,x_2,\ldots,x_n)^T \in
\mathcal{H}^n$.

In this paper we strengthen some results of Bouthat, Khare, Mashreghi,
and Morneau-Gu\'erin~\cite{Khare}. One of these computed the $\ell_p$
operator norm of all circulant matrices with non-negative entries. We now
extend this significantly:
\begin{itemize}
\item A twofold strengthening is that we work with Kronecker products $A
\otimes B$, where $A\in \mathcal{M}_n(\mathbb{C})$ and $B \in
\mathcal{B}(\mathcal{H})$ are arbitrary.

\item In this setting, we obtain lower and upper bounds for $w(A \otimes
B)$ and $\| A \otimes B \|_p$ for $p \in [1,\infty]$. Moreover, these
bounds are ``tight'' -- i.e., they coincide -- when $A$ is a dilation of
a doubly stochastic matrix (i.e.\ $A$ has all entries in $[0,\infty)$,
and all row and columns sums are equal). \blue{This includes all
circulant $A$ with entries in $[0,\infty)$, recovering the exact
calculations in~\cite{Khare}.}
\end{itemize}
 
\begin{theorem}\label{thm4-1}
Let $A=[a_{ij}]\in \mathcal{M}_n(\mathbb{C})$ 
and $B\in \mathcal{B}(\mathcal{H})$ be arbitrary. Then
\begin{equation}\label{E14}
w(A) w(B) \leq w(A \otimes B) \leq \| A \| w(B) \leq 
\left( \max_{1 \leq i \leq n} \sum_{j=1}^n |a_{ij}| \right)^{1/2} 
\left( \max_{1 \leq j \leq n} \sum_{i=1}^n |a_{ij}| \right)^{1/2} w(B).
\end{equation}
Next, given $1 \leq p \leq \infty$ we define $q$ via: $\frac{1}{p} +
\frac{1}{q} = 1$. Then,
\begin{equation}\label{E15}
\min_{1 \leq i \leq n} \left| \sum_{j=1}^n a_{ij} \right| \| B \|
\leq \| A \otimes B \|_p \leq
\left( \max_{1 \leq i \leq n} \sum_{j=1}^n |a_{ij}| \right)^{1/q} 
\left( \max_{1 \leq j \leq n} \sum_{i=1}^n |a_{ij}| \right)^{1/p}
\| B \|.
\end{equation}
\end{theorem}

As promised, we now record the tightness of these bounds for doubly
stochastic matrices $A$:

\begin{cor}\label{Copnorm}
If $B \in \mathcal{B}(\mathcal{H})$, and $A \in
\mathcal{M}_n(\mathbb{C})$ has non-negative real entries and is $k$ times
a doubly stochastic matrix for some \blue{$k \in [0,\infty)$}, then
$\|A\otimes B\|_p=k \|B\| \,\, \text{ and } \,\, w(A\otimes B)=k w(B).$
\end{cor}

For example, if $\mathcal{H} = \mathbb{C}$ and $B: \mathbb{C}\to
\mathbb{C}$ is the identity operator, then $ \| A\|_p=\| A\otimes B\|_p$
is the $\ell_p$ operator norm~\eqref{Eopnorm} of $A\in
\mathcal{M}_n(\mathbb{C})$. Thus Corollary~\ref{Copnorm} recovers this
norm for all rescaled doubly stochastic matrices, strictly subsuming the
circulant non-negative case in~\cite{Khare}.

\begin{proof}
The result is immediate if $k=0$ since $A=0$, so we assume $k>0$.
If $k^{-1}A$ is doubly stochastic, the first assertion is clear
from~\eqref{E15}, and the second assertion is equivalent to its special
case (for $B = (1) \in \mathcal{B}(\mathbb{C})$), i.e.\ that $w(A) = k$.
Now $(2k)^{-1} (A + A^*)$ is doubly stochastic, so its spectral radius is
at most $1$ by the Gershgorin circle theorem; as $1$ is an eigenvalue
(with eigenvector $(1,\dots,1)^T$), $r(A+A^*) = 2k$. Now use~\eqref{p1}
below.
\end{proof}

\subsection*{Organization of the paper}

  In Section~\ref{sec-ten}, we prove Theorems~\ref{th4} and~\ref{cor1}
and deduce other related results.
  In Section~\ref{sec-semi}, we extend these results for Kronecker
products, to the setting of a semi-Hilbertian space $\mathbb{C}^n\otimes
\mathcal{H}$, induced by the operator matrix $I_n\otimes P$ for arbitrary
positive $P \in \mathcal{B}(\mathcal{H})$.
  In Section~\ref{sec-had}, by applying numerical radius inequalities for
Kronecker products, we study numerical radius (in)equalities for the
Schur/entrywise product of matrices.
  In Section~\ref{sec-lpnorm}, we prove Theorem~\ref{thm4-1} (which yields
Corollary~\ref{Copnorm}). We also compute $\| A \otimes B \|_2$ for $A$ a
circulant matrix with diagonals $-a$ and off-diagonals $b$, for arbitrary
complex $a,b$ (extending the case of $a,b \in [0,\infty)$ in
\cite{Khare}).
  In Section~\ref{roots}, using the numerical radius of circulant
matrices, we obtain a new estimation formula for the roots of an
arbitrary complex polynomial.
  \blue{We end with some natural questions that arise from the results in
  this work, in the concluding Section~\ref{Sspeculation}.}

\section{Numerical radius inequalities for Kronecker products}\label{sec-ten}

In this section we obtain numerical radius bounds for $A\otimes B$ that
strengthen~\eqref{E0-1}, and then completely characterize the equality
of~\eqref{E0-1} when all $a_{ij} \in [0,\infty)$ (Theorems~\ref{th4}
and~\ref{cor1}).

Begin by noting from the definitions (see
Definition~\ref{D12}) that for any two Hilbert spaces
$\mathcal{K},\mathcal{H}$ and operators $A \in \mathcal{B}(\mathcal{K}),
B \in \mathcal{B}(\mathcal{H})$, we have
\[
\langle (A \otimes B)(x \otimes y), x' \otimes y' \rangle =
\langle (B \otimes A)(y \otimes x), y' \otimes x' \rangle, \quad \forall x,x'
\in \mathcal{K}, \ y,y' \in \mathcal{H}.
\]
Taking sums and limits, $w(A\otimes B)=w(B\otimes A)$. Thus,
from~\eqref{E0-1} (generalized to \cite[Equation (2)]{BhuniaPMSC}), we
obtain $w(A\otimes B) \leq  w(B)\|A\|.$ Also, it is easy to see that
$w(A\otimes B) \geq  w(A)w(B),$ see \cite{BhuniaPMSC}.
We collect all these inequalities together:
\begin{equation}\label{p3}
    w(A)w(B) \leq w(A\otimes B) \, \leq \, \min \left\{ w(A)\|B\|,
    w(B)\|A\|\right\}.
\end{equation}
In particular, if $w(A) = \| A \|$ or $w(B) = \| B \|$ then $w(A \otimes
B) = w(A) w(B)$.

We next record a ``$2 \times 2$ calculation'': $w \left(
\begin{bmatrix} 0 & \lambda \\ \mu & 0 \end{bmatrix} \right) =
\displaystyle \frac{|\lambda| + |\mu|}{2}$ for all $\lambda, \mu \in
\mathbb{C}$. This can be (indirectly) deduced from the results in
\cite{Hiz2011} and \cite{Bhunia-AdM24}; we now extend it to all
anti-diagonal matrices, show this from first principles, and use it below
without further mention:

\begin{prop}\label{P2x2}
Let $A = [ a_{ij} ]_{i,j=1}^n \in \mathcal{M}_n(\mathbb{C})$ be
anti-diagonal: $a_{ij} = {\bf 1}_{i+j=n+1} \lambda_i$ for all $i,j$. Then
\[
w(A) = \frac{1}{2} \max_{1 \leq j \leq \lceil n/2 \rceil} (|\lambda_j| +
|\lambda_{n+1-j}|).
\]
\end{prop}

\begin{proof}
We first prove the claimed bound is attained. Suppose the maximum is
attained at some $j$ (which is allowed to be the ``central'' position if
$n$ is odd). If $\lambda_j \lambda_{n+1-j} = 0$ then define $x = (x_1,
\dots, x_n) \in \mathbb{C}^n$ to have coordinates $x_j = x_{n+1-j} =
1/\sqrt{2}$ and all other $x_i = 0$. Then $|\langle Ax, x \rangle| =
\frac{|\lambda_j| + |\lambda_{n+1-j}|}{2}$, as desired. Otherwise
$\lambda_j \lambda_{n+1-j} \neq 0$; now choose $\theta,\mu \in [0,2\pi]$
such that $\frac{\lambda_j}{|\lambda_j|} e^{i\theta} =
\frac{\lambda_{n+1-j}}{|\lambda_{n+1-j}|} = e^{i \mu}$.
Let $x \in \mathbb{C}^n$ be such that $x_j = 1/\sqrt{2}$, $x_{n+1-j} =
e^{i\theta/2} / \sqrt{2}$, and all other $x_i = 0$. (There is no
ambiguity if $j = n+1-j$ is ``central'', since $\theta = 0$.) Then
\[
w(A) \geq |\langle Ax, x \rangle| = |e^{i \theta/2} e^{i \mu}| \cdot
\left| e^{i\theta/2} \lambda_j + e^{-i \theta/2} \lambda_{n+1-j} \right|
/ 2 = \left( |\lambda_j| + |\lambda_{n+1-j}| \right) / 2.
\]

To show the reverse inequality, compute for any $x = (x_1, \dots, x_n)^T
\in \mathbb{C}^n$:
\[
2 | \langle Ax, x \rangle| = 2 \left| \sum_{i=1}^n x_i \lambda_i
\overline{x}_{n+1-i} \right| \leq \sum_{i=1}^n 2 | x_i \lambda_i
\overline{x}_{n+1-i} | = \sum_{i=1}^n (|\lambda_i| + |\lambda_{n+1-i}|)
\cdot |x_i x_{n+1-i}|,
\]
using the triangle inequality. Now by the AM-GM inequality and choice of
$j$, this quantity is at most
$\displaystyle (|\lambda_j| + |\lambda_{n+1-j}|) \cdot \sum_{i=1}^n
(|x_i|^2 + |x_{n+1-i}|^2)/2 = 
(|\lambda_j| + |\lambda_{n+1-j}|) \cdot \| x \|^2$, as desired.
\end{proof}

\blue{We also provide a less computational proof of Proposition
\ref{P2x2}, based on the weakly unitarily invariance property of the
numerical radius.

\begin{proof}[Alternate proof of Proposition \ref{P2x2}] 
    Note that the matrix $A$ is permutationally similar to 
    $$\begin{bmatrix}
        0 & \lambda_1\\
        \lambda_n & 0
    \end{bmatrix} \oplus \ldots \oplus \begin{bmatrix}
        0 & \lambda_{\lfloor \frac{n-1}{2} \rfloor}\\
        \lambda_{n+1- \lfloor \frac{n-1}{2} \rfloor} & 0
    \end{bmatrix} \oplus A',$$
    where $A'=\begin{bmatrix}
        0 & \lambda_{n/2}\\
        \lambda_{(n/2)+1} & 0
    \end{bmatrix}$ if $n$ is even, and $A'=\begin{pmatrix}
         \lambda_{(n+1)/2}
    \end{pmatrix}$ if $n$ is odd. Now the result follows from $w(X\oplus
    Y)=\max \{w(X), w(Y) \}$ and the ``$2\times 2$ calculation'' above.
\end{proof}
}

We next study $w(A \otimes B)$ for operator matrices, where $A$ is
anti-diagonal and $B \in \mathcal{B}(\mathcal{H})$. One such result
\cite{Hiz2011}, used below, is that if $A = \begin{bmatrix} 0 & 1 \\
\lambda & 0 \end{bmatrix}$ with $|\lambda| = 1$ then $w(A \otimes B) =
w(A) w(B)$. We extend this to a larger class of anti-diagonal $A$ (whose
nonzero entries can differ in modulus):

\begin{cor}\label{lem3}
Suppose $A$ is a complex anti-diagonal matrix, and there exists an index
$1 \leq j \leq n$ such that $|\lambda_j| = |\lambda_{n+1-j}| \geq
|\lambda_i|$ for all other $i$. Then for any $B \in
\mathbb{B}(\mathcal{H})$, $w(A \otimes B) = w(A) w(B)$.
\end{cor}

\begin{proof}
Since $\| A \| = \max_i |\lambda_i| = |\lambda_j|$, and $w(A) =
|\lambda_j|$ by Proposition~\ref{P2x2}, the assertion follows
from~\eqref{p3} via sandwiching.
\end{proof}

\begin{remark}
Corollary~\ref{lem3} is atypical, in that for every anti-diagonal matrix
$A_0$ not in Corollary~\ref{lem3} (so that $0 < w(A_0) < \| A_0 \|$),
and every Hilbert space with dimension in $[2,\infty]$, the inequality
$w(A_0) w(B) \leq w(A_0 \otimes B)$ can be strict. Indeed, say
$\mathcal{H} = \mathbb{C}^2$ and
$B=\begin{bmatrix}
    0&1\\
    0&0
\end{bmatrix}.$
Then $A_0 \otimes B$ is again anti-diagonal, so Proposition~\ref{P2x2}
gives that
\[
w(B) = 1/2, \qquad w(A_0 \otimes B) = \| A_0 \|/2 = \| A_0 \| w(B) >
w(A_0) w(B).
\]
One can carry out a similar computation in any Hilbert space
$\mathcal{H}$ by choosing orthonormal vectors $u,v \in \mathcal{H}$ and
letting $B := u \langle -, v \rangle \in \mathcal{B}(\mathcal{H})$ be a
rank-one operator.
\end{remark}

Having discussed anti-diagonal matrices, we now continue towards refining
Holbrook's bound. We need the following operator norm inequality of Hou
and Du \cite{Hou_1995} for operator matrices.

\begin{lemma}[\cite{Hou_1995}]\label{lem2}
Let $P_{ij} \in \mathcal{B}(\mathcal{H})\ \forall 1 \leq i,j \leq n$ and
$\mathbb{P} = [P_{ij}]$. Then $\| \mathbb{P} \| \leq \big\|
[\|P_{ij}\|]_{n\times n} \big\|.$
\end{lemma}

Finally, it is well known (\cite[pp.~44]{Horn-2} and \cite{PMSC}) that if
$A\in \mathcal{M}_n(\mathbb{C})$ with all $a_{ij} \in [0,\infty)$, then
\begin{equation}\label{p1}
   w(A)=\frac12 w\left({A+A^*}\right)=\frac12 r\left({A+A^*}\right),
\end{equation} 
where $r(\cdot)$ denotes the spectral radius. This and the spectral
radius monotonicity of matrices with non-negative entries imply (see
\cite[pp.~491]{Horn2}) that for $A=[a_{ij}], A'=[a'_{ij}] \in
\mathcal{M}_n(\mathbb{C})$, 
\begin{eqnarray}\label{p2}
w(A) \leq w(A')\quad \text{whenever  $ 0 \leq a_{ij} \leq a'_{ij}$ for
all $i,j$}.
\end{eqnarray}

With these results in hand, we can now refine Holbrook's bound for
$w(A\otimes B)$:

\begin{proof}[Proof of Theorem~\ref{th4}]
From~\eqref{p3} and Proposition~\ref{P2x2}, we get
\begin{equation}\label{Etemp}
w\left(\begin{bmatrix}
         0& a_{ij}\\
         a_{ji}&0
     \end{bmatrix} \otimes B\right) \leq  w\left(\begin{bmatrix}
         0& a_{ij}\\
         a_{ji}&0
     \end{bmatrix}  \right) \|B\|= \frac12\left( |a_{ij}|+ |a_{ji}| \right) \|B\|.
\end{equation}

We now proceed. For any $\lambda \in \mathbb{C}$ with $|\lambda|=1,$ we have
    \begin{eqnarray*}
        \left \|\operatorname{Re}(\lambda A\otimes B  )\right \|
        &&= \left\| \left[ \frac{\lambda a_{ij}B+ \bar{\lambda} \bar{a}_{ji}B^*}{2} \right]_{n\times n} \right\|\\
        &&\leq \left\| \left[  \frac{ \| \lambda a_{ij}B+ \bar{\lambda} \bar{a}_{ji}B^* \| }{2} \right]_{n\times n} \right\| \quad (\text{by Lemma~\ref{lem2}})\\
        &&= w \left( \left[  \frac{ \| \lambda a_{ij}B+ \bar{\lambda} \bar{a}_{ji}B^* \| }{2} \right]_{n\times n} \right) \quad (\text{$w(D)=\|D\|$ if $D^* = D$})\\
                &&\leq  w \left( \left[ \max_{|\lambda|=1} \frac{ \| \lambda a_{ij}B+ \bar{\lambda} \bar{a}_{ji}B^* \| }{2} \right]_{n\times n} \right) \quad (\text{using~\eqref{p2}})\\
               & & =  w \left( \left[ w\left(\begin{bmatrix}
         0& a_{ij}B\\
         a_{ji}B&0
     \end{bmatrix} \right) \right]_{n\times n} \right)
    = w(C)\\
  &&  \blue{  \quad \left(\text{since $w(X)= \frac{1}{2} \sup_{|\lambda|=1} \left\| \lambda X+\bar{\lambda}X^* \right\|$ for every $X\in \mathcal{B}(\mathcal{H})$} \right) },\\
    \end{eqnarray*}
   where ${C}=[c_{ij}]$ with
      $c_{ij}=  \begin{cases}
	    |a_{ii}|w(B) & \text{if } i=j,\\
     w\left(\begin{bmatrix}
         0& a_{ij}\\
         a_{ji}&0
     \end{bmatrix} \otimes B\right)
     &\text{if } i\neq j.
	\end{cases}$
Maximize over $|\lambda|=1$ and use the simple fact that $w(A\otimes
B)=\underset{{|\lambda|=1}}{\max} \|\operatorname{Re}(\lambda A\otimes B  )\|$. Thus,
$w(A\otimes B)\leq w(C)$.
Now use~\eqref{p1}, \eqref{p2}, and~\eqref{Etemp} to get:
$w(A\otimes B) \leq  w(C) \leq w(C^\circ)$, where $C^\circ$ is as in the
theorem.    
\end{proof}

\begin{remark}
The inequality $w(A\otimes B) \leq w(C^\circ)$ also follows from the
bound $w\left([A_{ij}]_{i,j=1}^n\right)\leq
w\left([a_{ij}']_{i,j=1}^n\right)$, where   $A_{ij}\in
\mathcal{B}(\mathcal{H})$,  $a_{ii}'=w(A_{ii})$ if $i=j$ and
$a_{ij}'=\|A_{ij}\|$ if $i\neq j$ (see \cite[Theorem 1]{Abu} and
\cite{Bhunia-IJPA25}).
Along yet another approach: one might think of using Schur's
triangularization theorem to assume $A$ triangular, since the numerical
radius is weakly unitarily invariant. However, doing so would change the
entries of the matrix $A$ itself, and hence our bounds as well (which
explicitly use the entries of $A$).
\end{remark}

\begin{remark}
Suppose $A=[a_{ij}]\in \mathcal{M}_n(\mathbb{C})$ with $|a_{ij}|=|a_{ji}|
$ for all $i,j.$ Then using Theorem~\ref{th4},
$w(C) = w\left( [ |a_{ij}| ]_{n\times n} \right) w(B),$ where $C$ is
as in Theorem~\ref{th4}. This implies:
$w(A\otimes B) \leq    w\left( [ |a_{ij}| ]_{n\times n}  \right) w(B).$
    Now if $A_1 = \begin{bmatrix}
    0&1+i\\
    \sqrt{2}&0
\end{bmatrix}$ (where $i=\sqrt{-1}$) and $B=\begin{bmatrix}
    0&2\\
    0&0
\end{bmatrix}$,
then $w(C) = w\left( [ |a_{ij}| ]_{n\times n}  \right) w(B) =\sqrt{2}<
2\sqrt{2}=  w(A_1)\|B\|$.
Thus, while our result reduces to Holbrook's bound if $A$ is diagonal,
this example shows that our result can improve Holbrook's bound when one
considers the slightly larger class of normal matrices.
\end{remark}

Now Theorem~\ref{th4} yields the following corollary, which also refines
Holbrook's bound. 

\begin{cor}\label{th2}
Let $A=[a_{ij}]\in \mathcal{M}_n(\mathbb{C})$ and $B\in
\mathcal{B}(\mathcal{H}).$ Then the following inequalities hold:
\begin{enumerate}
\item $w(A\otimes B)  \leq w(A') w(B)$, where $A' =[a'_{ij}]$ with
$a'_{ij} = \max \left\{ |a_{ij}|, |a_{ji}| \right\}.$

\item $w(A\otimes B)  \leq w(\widehat{C})$, where
${\widehat{C}}=[\hat{c}_{ij}]$ with
\[
\hat{c}_{ij}=  \begin{cases}
|a_{ii}|w(B) & \text{ if } i=j,\\
\frac12 \left\| (|a_{ij}| |B|+|a_{ji}| |B^*|) \right\|^{1/2} \left\|(
|a_{ji}| |B|+|a_{ij}| |B^*| )\right\|^{1/2} & \text{ if } i\neq j.
\end{cases}
\]
\end{enumerate}
\end{cor}

\begin{proof}
\begin{enumerate}
\item From~\eqref{p3} we get
$$w\left(\begin{bmatrix}
         0& a_{ij}\\
         a_{ji}&0
     \end{bmatrix} \otimes B\right) \leq w(B) \left \|\begin{bmatrix}
         0& a_{ij}\\
         a_{ji}&0
     \end{bmatrix}  \right\|= \max \left\{ |a_{ij}|, |a_{ji}| \right\} w(B).$$
Therefore, Theorem~\ref{th4} together with~\eqref{p2} gives  $w(A\otimes
B) \leq w(C) \leq w(A') w(B).$

\item From \cite[Remark 2.7 (ii)]{Bhunia-AdM24}, we have
\[
w\left(\begin{bmatrix}
         0& a_{ij}\\
         a_{ji}&0
\end{bmatrix} \otimes B\right) \leq \frac12 \left\| (|a_{ij}|
|B|+|a_{ji}| |B^*|) \right\|^{1/2} \left\|( |a_{ji}| |B|+|a_{ij}|
|B^*|)\right\|^{1/2}.
\]
Hence, from Theorem~\ref{th4} and~\eqref{p2} we get $w(A\otimes B) \leq
w(C) \leq w(\widehat{C}).$ \qedhere
\end{enumerate}
\end{proof}

\begin{remark}
Clearly, $ \frac12 \left\| (|a_{ij}| |B|+|a_{ji}| |B^*|) \right\|^{1/2} \left\|( |a_{ji}| |B|+|a_{ij}| |B^*| )\right\|^{1/2} \leq \frac{|a_{ij}|+|a_{ji}|}{2}\|B\|.$
From~\eqref{p1}, it follows that if $a_{ij}\geq 0$ for all $i,j$, then 
$w(A)\|B\|=w\left( \left[\frac{a_{ij}+a_{ji}}{2}\|B\| \right]_{n\times n} \right).$
Therefore when $a_{ij}\geq 0$ for all $i,j$,
Corollary~\ref{th2}(2) also refines Holbrook's bound~\eqref{E0-1}
(via~\eqref{p2}).
\end{remark}

We now use Theorem~\ref{th4} to obtain a complete characterization for
the equality $w(A\otimes B)= w(A) \|B\|$, when all entries of $A$ are
non-negative. For this, we first show:

\begin{lemma}\label{need}
Let $A=[a_{ij}]\in \mathcal{M}_n(\mathbb{C})$ with $a_{ii}\neq 0$ for all
$i$, and let $\lambda \in \mathbb{C}.$  Then
\[
\left \| \begin{bmatrix}
 		\lambda a_{11}& a_{12} & \ldots& a_{1n} \\
 		a_{21} & \lambda a_{22}& \ldots& a_{2n} \\
 		\vdots&\vdots& \ddots & \vdots\\
 		a_{n1} & a_{n2} & \ldots& \lambda a_{nn}
 	\end{bmatrix}
 	\right\|  =\|A\| \, \text{ if and only if $\lambda =1.$}
\]
\end{lemma}

\begin{proof}
The sufficiency is trivial; to show the necessity, write $A=D+C$, where
$D=(1-\lambda) {\rm diag}(a_{11}, \dots, a_{nn})$ and $C$ is the matrix
on the left in the lemma.
From \cite[Theorem 8.13]{Zhang},
$\sigma_{\max}(C)+\sigma_{\min}(D)\leq \sigma_{\max}(C+D)\leq
\sigma_{\max}(C)+\sigma_{\max}(D),$ where
$\sigma_{\max}(\cdot)$ and $ \sigma_{\min}(\cdot)$ denote the maximum and
minimum singular values, respectively. Since
$\sigma_{\max}(C)=\sigma_{\max}(C+D)$ by the hypothesis,
$\sigma_{\min}(D)\leq 0$ and $\sigma_{\max}(D)\geq 0.$ As all $a_{ii}
\neq 0$, we obtain $\lambda=1$.
\end{proof}

\begin{proof}[Proof of Theorem~\ref{cor1}]
The sufficiency is trivial, from~\eqref{p3}. To show the necessity, let
$w(A\otimes B)= w(A) \|B\|$. Then from~\eqref{ECtilde} and the line
following it, we obtain $w(C^\circ) = w(A) \|B\|$ by sandwiching, with
$C^\circ$ as in~\eqref{ECtilde}. Using~\eqref{p1} twice, we have:
\[
w(C^\circ) = \frac{1}{2} r( C^\circ + (C^\circ)^* ) =
\frac{1}{2} \| C^\circ + (C^\circ)^* \|, \qquad w(A) \| B \| =
w( \| B \| A) = \frac{1}{2} \left\| \|B\| (A + A^*) \right\|.
\]
So these are equal; now using Lemma~\ref{need}, $\lambda = w(B) / \| B
\| = 1$.
\end{proof}

\section{Inequalities for Kronecker products in semi-Hilbertian spaces}\label{sec-semi}

The goal of this section is to record the extensions of the results
studied in Section~\ref{sec-ten}, to the setting of a semi-Hilbertian
space. To do this, first we need the following notations and
terminologies.
Let $\mathcal{H}$ be a complex Hilbert space and $P\in
\mathcal{B}(\mathcal{H}) $ be a nonzero positive operator, i.e.,
$\langle Px , x\rangle\geq 0$ for all $x\in \mathcal{H}.$ Consider the
semi-inner product $\langle \cdot, \cdot \rangle_P : \mathcal{H} \times
\mathcal{H} \to \mathbb{R}$  induced by $P$, namely, $\langle x ,
y\rangle_{P} =\langle Px , y\rangle$ for all $x,y\in \mathcal{H}.$ 
The semi-inner product $\langle \cdot, \cdot \rangle_P$ induces a
seminorm $\| \cdot \|_P$  on $\mathcal{H}$   given by
$\|x\|_P=\sqrt{\langle x,x\rangle_P}$ for all  $x\in \mathcal{H}$. This
makes $\mathcal{H}$ a semi-Hilbertian space, and $\|\cdot\|_P$ is a norm
on $\mathcal{H}$ if and only if $P$ is injective.

\begin{definition} 
An operator $B_1 \in\mathcal{B}(\mathcal{H})$ is called the $P$-adjoint
of $B \in \mathcal{B}(\mathcal{H})$ if for every $x,y\in \mathcal{H}$,
$\langle Bx , y\rangle_P=\langle x , B_1y\rangle_P$ holds, i.e., if $B_1$
satisfies the equation $PX=B^*P$.
\end{definition}
 
The set of all operators which admit $P$-adjoints is denoted by
$\mathcal{B}_{P}(\mathcal{H})$. From Douglas' theorem \cite{doug}, we get
$\mathcal{B}_{P}(\mathcal{H}) = \left\{ B\in
\mathcal{B}(\mathcal{H})\,:\;B^* \left(\mathcal{R}(P)\right)\subseteq
\mathcal{R}(P)\right\},$ where $\mathcal{R}(P)$ denotes the range of $P.$
For $B\in \mathcal{B}(\mathcal{H})$, the reduced solution of the equation
$PX=B^*P$ is a distinguished $P$-adjoint of $B$, which is denoted by
$B^{\sharp_P}$ and satisfies $B^{\sharp_P}=P^\dag B^*P$, where $P^\dag$
is the Moore--Penrose inverse of $P$. Again via Douglas' theorem, one can
show:
\begin{align*}
\mathcal{B}_{P^{1/2}}(\mathcal{H}) = &\ \left\{ B\in
\mathcal{B}(\mathcal{H})
\,:\;B^{*} \left(\mathcal{R}(P^{1/2})\right)\subseteq
\mathcal{R}(P^{1/2})\right\}\\
= &\ \{B \in \mathcal{B}(\mathcal{H}) : \exists~~ \lambda > 0
~~\mbox{such that}~~  \|Bx\|_P \leq \lambda \|x\|_P, \ \forall ~~x \in
\mathcal{H}\}.
\end{align*}
Here $\mathcal{B}_{P}(\mathcal{H})$ and
$\mathcal{B}_{P^{1/2}}(\mathcal{H})$ are two sub-algebras of
$\mathcal{B}(\mathcal{H})$ and
$\mathcal{B}_{P}(\mathcal{H})\subseteq\mathcal{B}_{P^{1/2}}(\mathcal{H})
\subseteq \mathcal{B}(\mathcal{H}).$  
The semi-inner product induces the $P$-operator seminorm on $\mathcal{B}_{P^{1/2}}(\mathcal{H})$, which is defined as
\[
\|B\|_P=\sup_{\substack{x\in \overline{\mathcal{R}(P)}\\
x\not=0}}\frac{\|Bx\|_P}{\|x\|_P}=\sup\{  \|Bx\|_{P}: { x\in \mathcal{H},
\, \|x\|_{P}= 1} \}.
\]

\begin{definition}
The $P$-numerical radius of $B \in \mathcal{B}_{P^{1/2}}(\mathcal{H})$,
denoted as $w_P(B)$, is defined as $ w_P(B)=\sup_{} \{  |\langle Bx ,
x\rangle_P| : {x\in \mathcal{H}, \|x\|_{P}= 1} \}.$
\end{definition}
	
If $P=I_{\mathcal{H}}$ (the identity operator on $\mathcal{H}$), then
$\|B\|_P=\|B\|$ and  $w_P(B)=w(B)$. It is well known that the
$P$-numerical radius $w_{P}(\cdot):
\mathcal{B}_{P^{1/2}}(\mathcal{H}) \to \mathbb{R}$ defines a seminorm and
is equivalent to the $P$-operator seminorm via the relation
$\frac{1}{2}\|B\|_P\leq w_P(B) \leq \|B\|_P.$

The semi-inner product $\langle\cdot ,\cdot\rangle_P$ induces an inner
product on the quotient space $\mathcal{H}/\mathcal{N}(P)$ defined as
$[\overline{x},\overline{y}] := \langle Px , y\rangle\ \forall x, y \in
\mathcal{H}$, where $\mathcal{N}(P)$ denotes the null space of $P$ and
$\overline{x}= x + \mathcal{N}(P)$ for $x \in \mathcal{H}$.
de Branges and Rovnyak  \cite{branrov} showed that
$\mathcal{H}/\mathcal{N}(P)$ is isometrically isomorphic to the Hilbert
space $\mathcal{R}(P^{1/2})$ with inner product
$[P^{1/2}x,P^{1/2}y] := \langle M_{\overline{\mathcal{R}(P)}}x ,
M_{\overline{\mathcal{R}(P)}}y\rangle,\;\forall\, x,y \in \mathcal{H}.$
Here $M_{\overline{\mathcal{R}(P)}}$ denotes the orthogonal projection
onto $\overline{\mathcal{R}(P)}$.

To present our results in this section, we now need the following known
lemmas, which give nice connections between $B\in
\mathcal{B}_{P^{1/2}}(\mathcal{H})$ and a certain operator
$\widetilde{B}$ on the Hilbert space $\mathcal{R}(P^{1/2})$.

\begin{lemma}[{\cite[Proposition 3.6]{acg3}}]\label{lem1-4}
Let $B\in \mathcal{B}(\mathcal{H})$ and let $Z_{P}: \mathcal{H}
\rightarrow \mathcal{R}(P^{1/2})$ be defined as $Z_{P}x = Px$ for all
$x\in \mathcal{H}$. Then $B\in \mathcal{B}_{P^{1/2}}(\mathcal{H})$ if and
only if there exists a unique operator $\widetilde{B}$ on
$\mathcal{R}(P^{1/2})$ such that $Z_P B =\widetilde{B}Z_P$.
\end{lemma}

From this one derives the next lemma:

\begin{lemma}\label{lem-4}
Let $B, B'\in \mathcal{B}_{P^{1/2}}(\mathcal{H})$ and let $\lambda\in
\mathbb{C}$ be any scalar. Then
\[
\widetilde{B+\lambda B'}=\widetilde{B}+\lambda
\widetilde{B'} ~~\mbox{and}~~
\widetilde{BB'}=\widetilde{B}\widetilde{B'}.
\]
\end{lemma}

\begin{lemma}[\cite{acg3, B-KPS, kais01}]\label{lem2-4}
Let $B\in \mathcal{B}_{P^{1/2}}(\mathcal{H})$. Then
$\|B\|_P=\|\widetilde{B}\|$ and $w_P(B)=w(\widetilde{B})$.
\end{lemma}

\begin{lemma}[\cite{BFP}]\label{tilde2020}
Let $\mathbb{P}= [P_{ij}]_{n \times n}$ be an $n\times n$ operator matrix
such that $P_{ij} \in \mathcal{B}_{P^{1/2}}(\mathcal{H})$ for all $i,j$.
Then $\mathbb{P} \in \mathcal{B}_{I_n\otimes P^{1/2}}(\mathcal{H}^n)$ and 
$\widetilde{\mathbb{P}}=[\widetilde{P_{ij}}]_{n \times n}.$
\end{lemma}

We can now present our results. Using Theorem~\ref{th4},
we show the following extension of it:

\begin{theorem}\label{th4-4}
     Let $A=[a_{ij}]\in \mathcal{M}_n(\mathbb{C})$ and $B\in \mathcal{B}_{P^{1/2}}(\mathcal{H})$. Then 
   $w_{I_n\otimes P}(A\otimes B)  \leq w(C),$
    where ${C}=[c_{ij}]$ with $c_{ij} = |a_{ii}|w_P(B)$ if
    $i=j$ and
     $w_{I_2\otimes P}\left(\begin{bmatrix}
         0& a_{ij}\\
         a_{ji}&0
     \end{bmatrix} \otimes B\right)$ otherwise.
\end{theorem}

\begin{proof}
Following Theorem~\ref{th4} together with Lemmas~\ref{lem1-4},
\ref{lem-4}, \ref{lem2-4} and~\ref{tilde2020},  we obtain that
         $w_{I_n\otimes P}(A\otimes B)= w(\widetilde{A\otimes B} )= w(A\otimes \widetilde{B} ) \leq w(C),$
    \text{where} ${C}=[c_{ij}]$ with $c_{ij}= |a_{ii}|w(\widetilde{B})$ if $i=j$
    and $w\left(\begin{bmatrix}
         0& a_{ij}\\
         a_{ji}&0
     \end{bmatrix} \otimes \widetilde{B}\right)$ otherwise.
  But then $C$ is as claimed.
\end{proof}

From the inequalities~\eqref{p3} and using Lemmas~\ref{lem1-4},
\ref{lem-4} and~\ref{lem2-4}, we obtain that   
\begin{equation}\label{p3-4}
w(A)w_P(B) \leq w_{I_n\otimes P}(A\otimes B) \leq  \min \left\{
w(A)\|B\|_P, w_P(B)\|A\|\right\}, \quad \forall A\in
\mathcal{M}_n(\mathbb{C}), B\in \mathcal{B}_{P^{1/2}}(\mathcal{H}).
\end{equation}

From these inequalities and using Theorem~\ref{th4-4}, we can deduce the
following.
 
\begin{cor}\label{th1-}
Let $A=[a_{ij}]\in \mathcal{M}_n(\mathbb{C})$ and $B\in
\mathcal{B}_{P^{1/2}}(\mathcal{H}).$  Then
\begin{enumerate}
    \item $w_{I_n \otimes P}(A\otimes B)  \leq  w\left( {\bf C}
 	\right),$ where $ {\bf C}=\begin{bmatrix}
 		|a_{11}|w_P(B)& |a_{12}| \|B\|_P& \ldots& |a_{1n}| \|B\|_P \\
 		|a_{21}| \|B\|_P& |a_{22}|w_P(B)& \ldots& |a_{2n}| \|B\|_P\\
 		\vdots&\vdots& \ddots & \vdots\\
 		|a_{n1}| \|B\|_P& |a_{n2}| \|B\|_P& \ldots& |a_{nn}|w_P(B)
 	\end{bmatrix}.$

  \item  Moreover, if all entries of $A$ are non-negative, then 
    \begin{eqnarray}\label{final}
         w_{I_n \otimes P}(A\otimes B)  \leq  w\left( {\bf C}
 	\right) \leq  w(A)\|B\|_P.
    \end{eqnarray}
\end{enumerate}
    \end{cor}

Finally, we completely characterize the equality in $w_{I_n \otimes
P}(A\otimes B) \leq w(A) \|B\|_P$ in~\eqref{p3-4}.

\begin{prop}\label{cor1-4-}
Let $A=[a_{ij}]\in \mathcal{M}_n(\mathbb{C})$ and $B\in \mathcal{B}_{P^{1/2}}(\mathcal{H}).$ 
Then  $w_{I_n \otimes P}(A\otimes B)= w(A) \|B\|_P$ if  $w_P(B)=\|B\|_P.$ Conversely,
if $a_{ij}\geq 0$ and $a_{ii}\neq 0$ for all $i,j$ and $w_{I_n \otimes P}(A\otimes B)= w(A) \|B\|_P$, then $w_P(B)=\|B\|_P.$
\end{prop}

\begin{proof}
These assertions follow from  the inequalities~\eqref{p3-4}
and~\eqref{final} and Lemma~\ref{need}.
\end{proof}

\section{Numerical radius \blue{(in)equalities} for Schur products
\blue{and powers}}\label{sec-had}
   
We now apply numerical radius inequalities for Kronecker
products to study the analogous (in)equalities for Schur/entrywise
products of matrices. (Thus, $\mathcal{H} = \mathbb{C}^n$ in this
section.) As the Schur product $A\circ B$ is a principal
submatrix of the Kronecker product $A\otimes B$, we first record:

\begin{lemma}\label{lem4}
    Let $A, B\in \mathcal{M}_n(\mathbb{C})$. Then
    $w(A \circ B)  \leq w(A\otimes B).$
\end{lemma}

This lemma (is well known, and) together with~\eqref{p3} gives
\begin{equation}\label{n-34}
w(A\circ B) \leq \min \{ w(A) \|B\|, w(B) \|A\|\} \leq 2w(A)w(B),
\qquad \forall A , B \in \mathcal{M}_n(\mathbb{C}).
\end{equation}
    
\blue{  Clearly, if $w(A)=\|A\|$} or $w(B)=\|B\|,$ then
        $w\left(A\circ B
 	\right) \leq   w(A) w(B)$
 (also see in \cite[ Corollary 4.2.17]{Horn-2}).
    Ando and Okubo proved \cite[Corollary 4]{Ando} that  if $A =[a_{ij}]\in \mathcal{M}_n(\mathbb{C})$ is positive semidefinite, then 
        $w\left(A\circ B
 	\right) \leq   \max_i \{ a_{ii}\} w(B).$
For another proof one can see \cite[Propsition 4.1]{Gau2016}. The
equality conditions of the above inequalities are studied in
\cite{Gau2016}. 

We begin by improving on the numerical radius inequalities~\eqref{n-34}.

\begin{prop}\label{th5}
Let $A,B \in \mathcal{M}_n(\mathbb{C}).$ Then $w (A \circ B) \leq w(C)$,
where $C$ is as in Theorem~\ref{th4}. The analogous result by switching
$A$ and $B$ also holds (by symmetry of the Schur product).
\end{prop}

\begin{proof}
We only show the first assertion: it follows from Theorem~\ref{th4} and
Lemma~\ref{lem4}.
\end{proof}

Using similar arguments as Corollary~\ref{th2}, from
Proposition~\ref{th5} we deduce the following.

\begin{cor}\label{th2--}
Let  $A =[a_{ij}],\, B \in \mathcal{M}_n(\mathbb{C}).$ Then 
$w\left(A\circ B \right)  \leq w(A') w(B),$ \text{where }
$A' =[a'_{ij}]$ with $a'_{ij}=  \max \left\{ |a_{ij}|, |a_{ji}|
\right\}.$
In particular, if $|a_{ij}|=|a_{ji}|$ for all $i,j$ then
$w\left(A\circ B \right) \leq w\left( [ |a_{ij}| ]_{n\times n}  \right)
w(B)$.
\end{cor}

Similar to Theorem~\ref{th4}, from Proposition~\ref{th5} we obtain the
following bounds.

\begin{cor}\label{th6}
Let $A, B \in \mathcal{M}_n(\mathbb{C}).$ Then $w(A \circ B) \leq
w(C^\circ)$, where $C^\circ$ is as in Theorem~\ref{th4}. The analogous
result by switching $A$ and $B$ also holds (by symmetry of the Schur
product).
\end{cor}

\begin{remark}
Clearly, if either $a_{ij}\geq 0$ or  $b_{ij}\geq 0$ for all $i,j$, then
Corollary~\ref{th6} gives a stronger upper bound than the one
in~\eqref{n-34}. 
\end{remark}

\noindent Using Corollary~\ref{th6} and proceeding as in
Theorem~\ref{cor1}, we deduce another equality-characterization:

\begin{prop}\label{cor2}
Let $ A=[a_{ij}],\, B\in \mathcal{M}_n(\mathbb{C})$ with $a_{ij}\geq 0$
and $a_{ii}\neq 0$ for all $i,j$.  Then 
$ w\left(A\circ B \right) = w(A) \|B\| \, \textit{ implies } \,
w(B)=\|B\|.$ 
\end{prop}  

\noindent However, the converse is not true in general.

We next derive numerical radius inequalities for Schur powers of
complex matrices $w(A^{\circ m})$.

\begin{prop}\label{th10}
Let $A=[a_{ij}]\in \mathcal{M}_n(\mathbb{C})$ and $m\in \mathbb{N}$.
Then:
\begin{enumerate}
\item $w(A^{\circ m}) \leq  w(A) \|A\|^{m-1} \leq 2^{m-1} w^m(A)$.
\item Moreover, if $w(A)=\|A\|$ (e.g.\ if $A$ is normal), then
$w(A^{\circ m}) \leq  w^m(A)$.

\item \blue{If $w(A) = \|A\|$ and $w(A^{\circ m}) = w^m(A)$ for some $m
\geq 1$, then}
\[
\blue{w(A^{\circ m}) = \| A^{\circ m} \| = \|A\|^m = w^m(A).}
\]
\end{enumerate}
\end{prop}

\blue{After writing the proof, we will characterize when equality holds
in the second part.}

\begin{proof}
(2)~is immediate from~(1). To show~(1), use Lemma~\ref{lem4}
and~\eqref{p3} to obtain: $w(A\circ B) \leq w(A\otimes B)\leq w(B) \|A\|$
for every $B\in \mathcal{M}_n(\mathbb{C})$. Successively letting $B = A,
A^{\circ 2}, \dots$ yields $w(A^{\circ m} ) \leq w(A) \|A\|^{m-1}.$ The
second inequality now holds because $\|A\| \leq 2 w(A)$.

\blue{Finally, (3)~follows from a chain of inequalities:
\begin{equation}
\|A\|^m = w(A)^m = w(A^{\circ m}) \leq \| A^{\circ m} \| \leq \|
A^{\otimes m} \| = \|A\|^m,
\end{equation}
where the second inequality holds because $A^{\circ m}$ is a submatrix of
$A^{\otimes m}$, hence of the form $P_1 A^{\otimes m} P_2$ for suitable
(projection) operators $P_1, P_2$.}
\end{proof}

We make the following remarks:
\begin{enumerate}
\item  The inequalities in Proposition~\ref{th10} are sharp. E.g.\ if
$A=\begin{bmatrix}
    0&\lambda\\
    0&0
\end{bmatrix}$
for any complex $\lambda$, then by Proposition~\ref{P2x2}, $w(A^{\circ
m}) = w(A) \| A \|^{m-1} = 2^{m-1} w(A)^m = |\lambda|^m/2$ for all $m
\geq 1$.

\item Proposition~\ref{th10} implies that if $w(A^{\circ m}) =  2^{m-1}
w^m(A)$ for some $m\in \mathbb{N}$, then $w(A)=\|A\| / 2.$ However,
the converse is not true in general; e.g.\ consider $A=\begin{bmatrix}
    1&1\\
    -1&-1
\end{bmatrix}.$
\end{enumerate}

Note that the inequality in Proposition~\ref{th10}(2) can be strict. For
example, let
 $A=\begin{bmatrix}
    a&b\\
    b&a
\end{bmatrix}$,
where $a,b>0.$ Then $w(A^{\circ m})=a^m+b^m< (a+b)^m=w^m(A)$ for every
integer $m \geq 2$. Thus, we now completely characterize \blue{the
equality in Proposition~\ref{th10}(2).

\begin{theorem}\label{Tref}
Let $m,n \geq 1$, and suppose $A \in \mathcal{M}_n(\mathbb{C})$ is such
that $w(A) = \| A \|$. Let the eigenvalues of $A$ be listed as $a_1,
\dots, a_n$ such that $|a_1| = \cdots = |a_k| > |a_{k+1}|, \dots, |a_n|$
for some $k \geq 1$.
\begin{enumerate}
\item The Jordan blocks of $A$ corresponding to the ``maximum-modulus
eigenvalues'' $a_1, \dots, a_k$ are of size $1 \times 1$.

\item Let $U = [u_1 | \cdots | u_k | B]$ be any unitary matrix such that
$U^* A U = {\rm diag}(a_1, \dots, a_k) \oplus A'$ for some $A' \in
\mathcal{M}_{n-k}(\mathbb{C})$.
Then $w(A^{\circ m}) = w^m(A)$ if and only if the subspace
\[
{\rm span}_{\mathbb{C}} \{ u_{j_1} \otimes \cdots \otimes u_{j_m} : 1
\leq j_1, \dots, j_m \leq k \} \ \bigcap \ {\rm span}_{\mathbb{C}} \{
{\bf e}_j^{\otimes m} : j = 1, \dots, n \}
\]
is nonzero and contains an eigenvector of $A^{\otimes m}$ --
equivalently, ${\rm span}_{\mathbb{C}} \{ {\bf e}_j^{\otimes m} : j = 1,
\dots, n \}$ contains an eigenvector of $A^{\otimes m}$ with eigenvalue
$\lambda$ such that $|\lambda| = |a_1|^m = \| A \|^m$.
\end{enumerate}

\noindent Moreover, if~(2) holds: $w(A^{\circ m}) = w^m(A)$, then the
same holds for every $1 \leq m' \leq m$.
\end{theorem}

Before proceeding further, we cite predecessors in the literature to the
first part.
\begin{enumerate}
\item After we had shown Theorem~\ref{Tref} {(following the referee's
suggestions)} while this manuscript was under revision, Tao pointed out
to us that Goldberg--Zwas~\cite{GZ} had shown Theorem~\ref{Tref}(1) in
the 1970s! Their proof does not use the numerical radius, so that our
proof of Theorem~\ref{Tref}(1) is different from theirs. We also remark
for completeness that Goldberg--Zwas obtain a \textit{characterization}
of when $w(A) = \|A\|$ (they term such matrices ``radial''). We do not
proceed along these lines, as our main focus was Theorem~\ref{Tref}(2).

\item Four decades after Goldberg and Zwas, Gau--Wu also characterized
when $w(A) = \|A\|$ in \cite[Proposition~2.2]{Gau2018}: this happens if
and only if $A$ is unitarily similar to $[a] \oplus B$, with $\|B\| \leq
|a|$ and $w(A) = \|A\| = r(A) = |a|$. It is not immediately clear if this
implies Theorem~\ref{Tref}(1), since one would need to check if $w(B) =
\|B\| = |a|$ in order to proceed inductively.

\item For completeness we also point out the paper~\cite{GTZ} by
Goldberg--Tadmor--Zwas, in which the authors characterize all complex
square matrices whose spectral and numerical radii agree.
The authors termed such matrices ``spectral'', following
Halmos~\cite{Halmos}.
\end{enumerate}

Returning to the proof of Theorem~\ref{Tref}, we begin with two
preliminary lemmas.

\begin{lemma}\label{lem-Ref}
Fix integers $1 \leq K \leq N$ and a matrix $X \in
\mathcal{M}_N(\mathbb{C})$. Let $T$ be the leading $K\times K$ principal
submatrix of $X$. Then $w(T)=\|X\|$ if and only if $X$ has an eigenvector
$x=\begin{bmatrix} y\\ {\bf 0}_{N-K}\end{bmatrix}$ with corresponding
eigenvalue $\lambda$ satisfying: $|\lambda|=\|X\|$ (and hence $Ty =
\lambda y$ too).
\end{lemma}

Here, $T$ is assumed to be a leading principal submatrix to facilitate
stating and proving the result with simpler notation. However, our proof
of Theorem~\ref{Tref} will use a variant of this result where $T =
A^{\circ m}$ is a non-leading principal submatrix of $X = A^{\otimes m}$.
Such a variant can be easily deduced from Lemma~\ref{lem-Ref}.

\begin{proof}
If $X$ has an eigenvector $x$ as specified, and we rescale $y$ to be unit
length (hence so is $x$), then
\[
w(T) \geq |\langle Ty,y \rangle| = |\langle Xx,x \rangle| = |\lambda| =
\| X \|.
\]
Moreover, $\|X\| \geq w(X) \geq w(T)$ (the latter by padding by zeros).
Hence $w(T) = \| X \|$.

Conversely, suppose $w(T) = |\langle Ty,y\rangle|$ for some unit vector
$y \in \mathbb{C}^K$. With $x$ the zero-padding of $y$, it follows by
Cauchy--Schwarz that
\[
\|X\| = w(T) = |\langle Ty,y \rangle| = |\langle Xx,x \rangle| \leq
\|Xx\| \|x\| \leq \|X\|.
\]

Hence $x$ and $Xx$ are linearly dependent, so $x$ is an eigenvector of
$X$ (and hence, $y$ of $T$), and the corresponding eigenvalue $\lambda =
\langle Xx,x \rangle$ satisfies: $|\lambda| = \|X\|$.
\end{proof}

\begin{remark}
We stress -- from the above proof -- that if $w(T) = \|X\|$, then
\textit{any} unit vector $y \in \mathbb{C}^K$ with $w(T) = | \langle Ty,y
\rangle|$ is an eigenvector of $T$ (and its zero-padding $x$ of $X$),
with common eigenvalue $\lambda$ such that $|\lambda| = \|X\|$.
\end{remark}

In order to characterize when $w(A^{\circ m}) = w(A)^m = w(A^{\otimes
m})$ (by~\eqref{p3}), we need to work with $A^{\circ m}$ as a principal
submatrix of $A^{\otimes m}$. Thus, we quickly record the coordinates in
which $A^{\circ m}$ sits inside $A^{\otimes m}$. More generally, we have
(for any tuple of equidimensional matrices over any field):

\begin{lemma}\label{Lpos}
Given integers $p,q,m \geq 1$ and $p \times q$ matrices $A_1, \dots,
A_m$, their Schur product $A_1 \circ \cdots \circ A_m$ occurs as a
submatrix of their tensor product $\otimes_{i=1}^m A_i$, in the rows
numbered $\{ j(1 + p + p^2 + \cdots + p^{m-1}) + 1 : 0 \leq j < p \}$
and the columns
numbered $\{ j(1 + q + q^2 + \cdots + q^{m-1}) + 1 : 0 \leq j < q \}$.
\end{lemma}

In other words, the row numbers are one more than non-negative integer
multiples (in base $p$) of $(11\dots1)_p$; and similarly for the column
numbers.

With these lemmas at hand, we have:

\begin{proof}[Proof of Theorem~\ref{Tref}]\hfill
\begin{enumerate}
\item If $w(A) = \|A\| =0$, then $A=0$ and the result is immediate. Else:
as $w(A) \geq |a_j| \ \forall j \leq n$, applying Lemma~\ref{lem-Ref}
with $T=X$ shows that $\| A \| = |a_1| = \cdots = |a_k|$. We now suppose
for some $0 \leq l < k$ that
\begin{equation}\label{Einduct}
U_l^* A U_l = {\rm diag}(a_1,
\dots, a_l) \oplus A'_l, \qquad \text{with } U_l \text{ unitary},
\end{equation}
and proceed inductively on $l$. (If $l=0$ then let $U_0 = {\rm
Id}_n$.) Write $U_l = [u_1 | \cdots | u_l | B']$; then one verifies that
$A u_j = a_j u_j$ for all $1 \leq j \leq l$. Since $l < k$, $A'$ has an
eigenvalue $a_{l+1}$ (with $|a_{l+1}| = \| A \|$) with a unit-length
eigenvector $v'$ -- and $U_l^* A U_l$ has an associated eigenvector
$\begin{bmatrix} {\bf 0}_l \\ v' \end{bmatrix}$. This eigenvector is
orthogonal to all ${\bf e}_j = U_l^* u_j$ (for $1 \leq j \leq l$), so
$u_{l+1} := U_l \begin{bmatrix} {\bf 0}_l \\ v' \end{bmatrix}$ is
orthogonal to $U_l {\bf e}_j = u_j$ for all $j \leq l$.

Now let $U_{l+1} = [u_1 | \cdots | u_{l+1} | B]$ be any unitary matrix
(so $a_j B^* u_j = B^* A u_j = 0\ \forall j \leq l+1$). An explicit
computation then reveals:
\[
U_{l+1}^* A U_{l+1} = \begin{bmatrix}
a_1 & \cdots & 0 & u_1^* A B \\
\vdots & \ddots & \vdots & \vdots\\
0 & \cdots & a_{l+1} & u_{l+1}^* A B \\
{\bf 0} & \cdots & {\bf 0} & B^* A B
\end{bmatrix}
\]

For $1 \leq j \leq l+1$, let $r_j \neq 0$ denote the $j$th row of this
matrix. Then
\[
|a_j| = \| U_{l+1}^* A U_{l+1} \| \geq \| v_j \| \geq |\langle v_j, {\bf
e}_j \rangle|, \qquad \text{where } v_j := (U_{l+1}^* A U_{l+1}) \cdot
\textstyle \frac{r_j^*}{\|r_j^*\|}.
\]

But $| \langle v_j, {\bf e}_j \rangle| = \| r_j^* \|
= \sqrt{|a_j|^2 + \| B^* A^* u_j \|^2}$, 
so $u_j^* AB = 0$ and we obtain~\eqref{Einduct} for $l+1$, as desired.
It follows by induction on $l$ that $A$ is unitarily similar to ${\rm
diag}(a_1, \dots, a_k) \oplus A'$. We are now done by the uniqueness of
the Jordan normal form.

Moreover, the eigenvalues of $A'$ are necessarily $a_{k+1}, \dots, a_n$,
and all of them are $< |a_1| = \| A \|$. Thus Lemma~\ref{lem-Ref}
(modified to work with the trailing principal submatrix) shows that
$w(A') < \| A \|$.

\item Let $U$ be as in the statement (it exists by part~(1)). First note
for the matrix $A^{\otimes m}$ that
\[
w(A^{\otimes m})=w^m(A)=\|A\|^m=\|A^{\otimes m}\|,
\]
e.g.\ by~\eqref{p3}. Moreover, its maximum-modulus eigenvalues are
precisely the $k$-fold products
\begin{equation}\label{Efamily}
a_{j_1} \cdots a_{j_m}, \quad 1\leq j_1,\ldots, j_m\leq k,
\end{equation}
and the associated unit-norm eigenvectors are $u_{j_1} \otimes \cdots
\otimes u_{j_m}$.

Note that $A^{\circ m}$ is a principal submatrix of $A^{\otimes m}$, with
row and columns corresponding to the positions of the nonzero entries
in ${\bf e}_j^{\otimes m}$ for $1 \leq j \leq n$ (this latter follows
from Lemma~\ref{Lpos}, wherein we set $p=n, q=1$).
Thus, a ``non-leading-yet-principal'' version of Lemma~\ref{lem-Ref}
implies that $w(A^{\circ m})=w^m(A) = \| A^{\otimes m}\|$ if and only if
$A^{\otimes m}$ has an eigenvector in ${\rm span}_{\mathbb{C}}
\{{\bf e}_j^{\otimes m} : j=1,\ldots,n \}$, with associated eigenvalue
$\lambda$ satisfying
$|\lambda|=\left\|A^{\otimes m} \right\| = |a_1|^m$. By choice of $k$ in
the hypotheses, $\lambda$ lies in the collection~\eqref{Efamily}.
\end{enumerate}

Finally, suppose $m > 1$ and $w(A) = \|A\|, w(A^{\circ m}) = w^m(A)$. By
downward induction, it suffices to show that $w(A^{\circ (m-1)}) =
w^{m-1}(A)$. Now since $A^{\circ m}$ is a principal submatrix of $A^{\circ (m-1)} \otimes A$, by padding by zeros and~\eqref{p3} we have: 
\[
w^m(A) = w(A^{\circ m}) \leq w(A^{\circ (m-1)} \otimes A) = w(A^{\circ
(m-1)}) w(A) \leq w^m(A),
\]
where the final inequality is by Proposition~\ref{th10}(2). Hence all
inequalities are equalities, and from the final step we get
$w(A^{\circ(m-1)}) = w^{m-1}(A)$.
\end{proof}

\begin{example}
Theorem~\ref{Tref}(1) shows that if $w(A) = \|A\|$, then we have the
decomposition $U^* A U = {\rm diag}(a_1,\dots,a_k) \oplus A'$, where all
eigenvalues of $A'$ are strictly less than $|a_1| = \cdots = |a_k|$. The
proof of this part also showed that $w(A') < w(A) = \|A\|$. It is natural
to ask if $\|A'\|<\|A\|$ or not. This turns out to be not always the
case; for instance, let $A = \begin{pmatrix} 1 & 0 & 0 \\ 0 & 0 & 1 \\ 0
& 0 & 0 \end{pmatrix}$. Then $w(A) = 1 = \|A\|$, and $A' = E_{12}$, so
$w(A') = 1/2$ and $\|A'\| = 1$ (and all eigenvalues of $A'$ are zero).
\qed
\end{example}

\begin{remark}
In the concluding section~\ref{Sspeculation}, we will list some questions
that naturally arise from Theorem~\ref{Tref}.
\end{remark}

We next reformulate the characterization in Theorem~\ref{Tref}(2) as
follows.

\begin{cor}\label{Cref}
Setting as in Theorem~\ref{Tref}, and let $U$ be as in part~(2) there.
Let $D := U^* A U = {\rm diag}(a_1, \dots, a_k) \oplus A'$, and let $V_k
:= {\rm span}_{\mathbb{C}}({\bf e}_1, \dots, {\bf e}_k) \subseteq
\mathbb{C}^n$.
Then $w(A^{\circ m}) = w^m(A)$ if and only if $D^{\otimes m}$ has an
eigenvector in $V_k^{\otimes m} \cap {\rm span}_{\mathbb{C}} \{ (U^*
{\bf e}_j)^{\otimes m} : 1 \leq j \leq n \}$, with eigenvalue $\lambda$
such that $|\lambda| = |a_1|^m$.
\end{cor}

\begin{proof}
This follows from Theorem~\ref{Tref}(2) at once, using that
(a)~the Kronecker product is multiplicative, so that $(AB)^{\otimes m} =
A^{\otimes m} B^{\otimes m}$ for all compatible matrices (or vectors)
$A,B$ and integers $m \geq 1$;
and (b)~$U^* u_j = {\bf e}_j$ for all $n \times n$ unitary matrices $U^*$
and all $1 \leq j \leq n$.
\end{proof}

In Corollary~\ref{Cref}, we can also recover additional information about
the eigenvector of $D^{\otimes m}$.

\begin{prop}
Setting as in Corollary~\ref{Cref}.
If $w(A^{\circ m})=w^m(A)$,  then $\oplus_{j=1}^k y_j \oplus { \bf
0}_{(n-k)n^{m-1}} $ is a unit eigenvector of $D^{\otimes m}$,  with
eigenvalue $\lambda$ such that $|\lambda|=|a_1|^m$ and each $y_j=
\oplus_{l=1}^ky_{jl}\oplus {\bf 0}_{(n-k)n^{m-2}}$ is an eigenvector of
$D^{ \otimes m-1}$, with eigenvalue $\frac{1}{ \|y_j\|^2}\langle ( D^{
\otimes m-1}) y_j,y_j\rangle=e^{-i \theta_j}w^{m-1}(A)$ where $\theta_j=
\theta+ \text{arg}\,(a_j)$ for some $\theta \in \mathbb{R}$.
\end{prop}

\begin{proof}
Throughout this proof, $x = (x_1, \dots, x_n)^T$ denotes a unit length
vector in $\mathbb{C}^n$. Choose $x$ such
that
$|\langle (A^{\circ m}) {x},
{x}\rangle|=w(A^{\circ m}).$ We have
    \begin{eqnarray*}
        w^m(A) &=&| \langle (A^{\circ m}) {x}, {x}\rangle |
        = |\langle ( A^{ \otimes m})x',x' \rangle | \quad (\textit{where  $x' = \sum_{j=1}^n x_j
\left(e_j^{ \otimes m}\right) \in \mathbb{C}^{n^m}$}) \\
        & =& |\langle (D^{ \otimes m}) ( {U^*}^{ \otimes m})x',( {U^*}^{ \otimes m})x' \rangle |  \\
        &=&|\langle (D^{ \otimes m})y,y \rangle |
        \quad (\textit{where } y= ({U^*}^{ \otimes m})x'=\oplus_{j=1}^k y_j \oplus { y'}_{(n-k)n^{m-1}} \in
\mathbb{C}^{n^m})\\
        &=&  |\langle \left( \oplus_{j=1}^k a_j (D^{ \otimes m-1}) \oplus (A'\otimes D^{ \otimes m-1}) \right) \oplus_{i=1}^ky_i\oplus y',\oplus_{i=1}^ky_i\oplus y' \rangle |\\
        &=& \left| \sum_{j=1}^ka_j\langle (D^{ \otimes m-1})y_j,y_j\rangle + \langle A'\otimes D^{ \otimes m-1}y',y'\rangle \right| \\
        &\leq & \left| \sum_{j=1}^ka_j\langle (D^{ \otimes m-1})y_j,y_j\rangle \right| +\left|\langle A'\otimes D^{ \otimes m-1}y',y'\rangle \right| \\
         &\leq &  \sum_{j=1}^k \left| a_j\langle (D^{ \otimes m-1})y_j,y_j\rangle \right| + \left|\langle A'\otimes D^{ \otimes m-1}y',y'\rangle \right| \\
         &\leq& \sum_{j=1}^k |a_j| w^{m-1}(A)\left\| y_j \right\|^2 + w(A' \otimes D^{ \otimes m-1})\left\| y' \right\|^2\\
         &\leq& \sum_{j=1}^k |a_j| w^{m-1}(A)\left\| y_j \right\|^2 + w(A') w^{m-1}(A)\left\| y' \right\|^2
         \quad (\text{since}\ w(D^{ \otimes m-1})=w^{m-1}(A)) \\
         &\leq& w^m(A)\left(\sum_{j=1}^k \left\| y_j \right\|^2+\|y'\|^2\right) \quad (\text{since}\ w(A')\leq w(A))\\
         &=& w^m(A) \|y\|^2= w^m(A)\|x'\|^2=w^m(A) \|{x}\|^2
         = w^m(A).
    \end{eqnarray*}
    Thus the above are all equalities. Since
    $w(A') < |a_1|=w(A),$  $y'=0$; and so $ 
         \left| \sum_{j=1}^ka_j\langle (D^{ \otimes m-1})y_j,y_j\rangle \right| 
         = \sum_{j=1}^k \left| a_j \right| \left|\langle (D^{ \otimes m-1})y_j,y_j\rangle \right| $ $ 
         = \sum_{j=1}^k |a_j| w^{m-1}(A)\left\| y_j \right\|^2.$ Hence
\[\langle (D^{ \otimes m-1}) y_j,y_j\rangle=e^{-i \theta_j}w^{m-1}(A)
\|y_j\|^2 \quad \text{ with $\theta_j= \theta+ \text{arg }(a_j)$ for some
$\theta \in \mathbb{R}.$}
\]
From the above, we also have $|\langle (D^{ \otimes m})y,y \rangle | =\|(D^{ \otimes m})y\|  \|y\|=|a_1|^m$ and  $|\langle (D^{ \otimes m-1})y_i,y_j \rangle | =\|(D^{ \otimes m-1})y_j\|  \|y_j\|= w^{m-1}(A)\|y_j\|^2.$ This concludes that
 $y$ and $y_j$ are eigenvectors of $D^{ \otimes m}$ and $D^{ \otimes m-1}$, respectively.  Again, using a similar approach we can show $y_j= \oplus_{l=1}^ky_{jl}\oplus {\bf 0}_{(n-k)n^{m-2}}.$
\end{proof}
}

\section{Numerical radius and $\ell_p$ operator norm \blue{inequalities}
for Kronecker products} \label{sec-lpnorm}

We now show Theorem~\ref{thm4-1} (and hence Corollary~\ref{Copnorm} for
doubly stochastic matrices), leading to refinements/extensions of results
of Bouthat, Khare, Mashreghi, and Morneau-Gu\'erin~\cite{Khare}.

\begin{proof}[Proof of Theorem~\ref{thm4-1}]
We begin by showing~\eqref{E15}. Let $p \in [1,\infty)$.
For $x=(x_1,x_2,\ldots,x_n)^T \in \mathcal{H}^n$,
\begin{eqnarray*}
   \| (A\otimes B)x\|_p^p = \sum_{i=1}^n \left\| \sum_{j=1}^na_{ij}Bx_j \right\|^p
    &\leq& \sum_{i=1}^n  \left( \sum_{j=1}^n|a_{ij}| \left\| Bx_j \right\|\right)^p
    \leq \left\| B \right\|^p \sum_{i=1}^n  \left( \sum_{j=1}^n|a_{ij}|  \|x_j\| \right)^p\\
     &\leq& \left\| B \right\|^p  \sum_{i=1}^n \left(\sum_{j=1}^n|a_{ij}|\right)^{p/q}    \left( \sum_{j=1}^n   |a_{ij}| \|x_j\|^p \right) \quad (\text{H\"older})\\
     &\leq& \left\| B \right\|^p \left( \max_{1 \leq i \leq n} \sum_{j=1}^n |a_{ij}| \right)^{p/q}  \sum_{i=1}^n     \left( \sum_{j=1}^n   |a_{ij}| \|x_j\|^p \right) \\
      &\leq& \left\| B \right\|^p \left( \max_{1 \leq i \leq n} \sum_{j=1}^n |a_{ij}| \right)^{p/q}  \sum_{j=1}^n     \left( \|x_j\|^p \sum_{i=1}^n   |a_{ij}|  \right) \\
      &\leq& \left\| B \right\|^p \left( \max_{1 \leq i \leq n} \sum_{j=1}^n |a_{ij}| \right)^{p/q} \left( \max_{1 \leq j \leq n} \sum_{i=1}^n |a_{ij}| \right)^{}   \sum_{j=1}^n      \|x_j\|^p\\
      &=& \left\| B \right\|^p \left( \max_{1 \leq i \leq n} \sum_{j=1}^n |a_{ij}| \right)^{p/q} \left( \max_{1 \leq j \leq n} \sum_{i=1}^n |a_{ij}| \right)^{}       \|x\|^p_p.
\end{eqnarray*}

Take the $p$th root, divide by $\| x \|_p$ (for $x \neq 0$), and take the
supremum over all $x \neq 0$, to obtain $\| A\otimes B\|_p \leq \left\| B
\right\| \left( \max_{1 \leq i \leq n} \sum_{j=1}^n |a_{ij}|
\right)^{1/q} \left( \max_{1 \leq j \leq n} \sum_{i=1}^n |a_{ij}|
\right)^{1/p} .$

\blue{
We now show $\| A\otimes B\|_p \geq \| B \| \min_{1 \leq i \leq n} \left|
\sum_{j=1}^n a_{ij} \right|.$ Since $B\in
\mathcal{B}(\mathcal{H}),$ there exists a sequence
$\{x_m\}_{m=1}^{\infty}$ in $\mathcal{H} \setminus \{ 0 \}$
with $\lim_{m \to \infty } \frac{\|Bx_m\|}{\|x_m\|}= \|B\|$. Letting
$z_m=(x_m,x_m,\ldots,x_m)^T \in \mathcal{H}^n$,
\begin{eqnarray*}
    \| A\otimes B\|_p^p \geq \lim_{m\to \infty} \frac{ \| (A\otimes B)z_m\|_p^p}{\|z_m\|_p^p}
     &=& \lim_{m\to \infty} \frac{\sum_{i=1}^n  \left( \left\| \sum_{j=1}^na_{ij}  Bx_m \right\|\right)^p}{\|z_m\|_p^p}\\
      &=& \lim_{m\to \infty} \frac{\sum_{i=1}^n  \left(  |\sum_{j=1}^na_{ij} | \left\| Bx_m \right\|\right)^p}{\|z_m\|_p^p}\\
    &\geq&  \min_{1 \leq i \leq n} \left| \sum_{j=1}^n a_{ij} \right|^p    \lim_{m\to \infty}   \frac{  \sum_{i=1}^n    \left\| Bx_m \right\|^p}{\|z_m\|_p^p}\\
	&=&\min_{1 \leq i \leq n} \left| \sum_{j=1}^n a_{ij} \right|^p
	\frac{  \|B\|^p \sum_{i=1}^n \left\| x_m
	\right\|^p}{\|z_m\|_p^p}
	= \|B\|^p \min_{1 \leq i \leq n} \left|
	\sum_{j=1}^n a_{ij} \right|^p.
\end{eqnarray*}
}
Hence, $\| A\otimes B\|_p \geq
\| B \| \min_{1 \leq i \leq n} \left| \sum_{j=1}^n a_{ij} \right|$.
Similar arguments help show these bounds for $p = \infty$.

Now we show~\eqref{E14}. Using~\eqref{p3}, we have $w(A) w(B) \leq
w(A\otimes B) \leq \|A\| w(B)$. To complete the proof we need to show
that
$\|A\| \leq \left( \max_{1 \leq i \leq n} \sum_{j=1}^n |a_{ij}|
\right)^{1/2} \left( \max_{1 \leq j \leq n} \sum_{i=1}^n |a_{ij}|
\right)^{1/2}$;
but this follows by taking $p=2$ and $B: \mathbb{C}\to \mathbb{C}$ to be
the identity operator in~\eqref{E15}.
\end{proof}

\subsection{Further (twofold) extensions}

The next result from \cite{Khare} that we improve involves a special
class of doubly stochastic matrices. Recall that a circulant matrix
$Circ(a_1,a_2,a_3\ldots,a_n )$ is
\[
Circ(a_1,a_2,a_3,\ldots,a_n)=\begin{bmatrix}
     a_{1}& a_{2} & a_3 & \ldots& a_{n} \\
 		a_{n} & a_1&  a_{2}& \ldots& a_{n-1} \\
   	a_{n-1} & a_n&  a_{1}& \ldots& a_{n-2} \\
 		\vdots&\vdots& \vdots& \ddots & \vdots\\
 		a_2& a_{3} & a_{4} & \ldots&  a_{1}
\end{bmatrix}.
\]
It was shown in \cite[Theorem 4.1]{Khare} that for scalars $a,b\in [0,
\infty)$,
 \begin{eqnarray}\label{--2}
     \|Circ(-a,b,b,\ldots,b)\|_2= \begin{cases}
         a+b & \text{ if $(n-2)b\leq 2a$},\\
          (n-1)b-a & \text{ if $(n-2)b\geq 2a$}.
     \end{cases}
 \end{eqnarray}

We now provide a twofold extension (as well as a simpler proof)
of~\eqref{--2}: the scalars $a,b$ need not be non-negative or even real;
and instead of just the matrix $A=Circ(-a,b,b,\ldots,b)$ we work with $A
\otimes B$ for arbitrary $\mathcal{H}$ and $B \in
\mathcal{B}(\mathcal{H})$. (Note below that $\| B \|_2 = \| B \|$ for all
$\mathcal{H}$, $B \in \mathcal{B}(\mathcal{H})$.)

\begin{theorem}\label{thm4-2}
Fix any $a,b \in \mathbb{C}$ and $B\in \mathcal{B}(\mathcal{H})$, and let
$A=Circ(-a,b,b,\ldots,b)$. Then:
\begin{enumerate}
    \item $\|A\otimes B\|_2= \max \{|a+b|, |(n-1)b-a| \} \|B\|.$

    \item  $w(A\otimes B)=  \max \{|a+b|, |(n-1)b-a| \} w(B).$
\end{enumerate}
    \end{theorem}
    
In particular, setting $B: \mathbb{C}\to \mathbb{C}$ as the
identity, we get: $\|Circ(-a,b,b,\ldots,b)\|_2\ \forall a,b\in
\mathbb{C}.$

\begin{proof}
Here $A=-(a+b) I_n + b {\bf 1}_{n\times n}$ is normal, with $I_n, {\bf
1}_{n \times n} \in \mathcal{M}_n(\mathbb{C})$ the identity and all-ones
matrix, respectively. Hence $\max \{|a+b|, |(n-1)b-a|
\}=r(A)=w(A)=\|A\|$. Thus $\|A\otimes B\|_2 = \|A\| \|B\|$
and $w(A\otimes B)= w(A)w(B)$ are as claimed.
\end{proof}

Our second twofold extension is of a result in \cite{Khare} that computes
$\| A \otimes B \|_2$. We extend from real to complex matrices, and also
using any $\mathcal{H}$ and $B \in \mathcal{B}(\mathcal{H})$ (vis-a-vis
$\mathcal{H} = \mathbb{C}$ and $B = (1)$ in \cite{Khare}).

\begin{prop}\label{thm4-3}
Let $A=[c_1|c_2|\ldots|c_n]\in \mathcal{M}_n(\mathbb{C})$, whose columns
satisfy $c_i^*c_i=\alpha$ and $c_i^*c_j=\beta$ (for all $i<j$) for some
real scalars $\alpha, \beta$, and let $B\in \mathcal{B}(\mathcal{H}).$
Then
\[
\| A\otimes B\|_2=\max \left\{\sqrt{|\alpha-\beta|},
\sqrt{|\alpha+(n-1)\beta|} \right\} \|B\|.
\]
\end{prop}

\begin{proof}
As $\|A\otimes B\|_2=\|A\| \|B\|$, it suffices to show the result for
$\mathcal{H} = \mathbb{C}$ and $B = (1)$. Clearly, $A^*A = Circ( \alpha,
\beta, \beta, \ldots, \beta) =(\alpha-\beta) I_n+\beta {\bf 1}_{n\times
n}$.  Therefore, $\|A\|=\|A^*A\|^{1/2}=r^{1/2}(A^*A)=\max
\left\{\sqrt{|\alpha-\beta|}, \sqrt{|\alpha+(n-1)\beta|} \right\}.$
\end{proof}

If $A=Circ(a_1,a_2,a_3)$ with all $a_j$ real, $A$ satisfies the
conditions of Proposition~\ref{thm4-3}, and so:

\begin{cor}\label{corr2}
Let  $A=Circ(a_1,a_2,a_3)\in \mathcal{M}_3(\mathbb{C})$, where
$a_1,a_2,a_3\in \mathbb{R}$. If $B\in \mathcal{B}(\mathcal{H}),$ then
\[
\|A\otimes B\|_2=\max \left\{{\left|a_1+a_2+a_3\right|}, \sqrt{ \left|
a_1^2+a_2^2+a_3^2-  (a_1a_2+a_2a_3+a_1a_3)\right| } \right\} \|B\|.
\]
\end{cor}

Our third and final (threefold) extension, is of bounds in~\cite{Khare}
for $\| A \otimes B \|_p$ where $p \neq 2$. It is shown in
\cite[Theorem~5.1 and Section~6]{Khare} that if $a,b \in [0,\infty)$ and
$n \geq 2$, then
\begin{equation}\label{--3}
     \max \{a+b, |(n-1)b-a| \} 
\leq \|Circ(-a,b,b,\ldots,b)\|_p \leq a + b + nb \kappa, 
     \qquad \forall p \in [1,\infty],
\end{equation}
where (as we show below) $\kappa \in (1,\infty)$. We extend this in our
next result, with
(i)~$a,b \in \mathbb{C}$ allowed to be arbitrary;
(ii)~the upper bound replaced by a quantity that is at most $|a + b| +
n|b|$ (so $\kappa$ is also replaced, by $1$); and
(iii)~$A = Circ(-a,b,b,\dots,b)$ replaced by $A \otimes B$ for arbitrary
$B \in \mathcal{B}(\mathcal{H})$.
 
\begin{theorem}\label{Tfinal}
Let $a,b \in \mathbb{C}$, $n \geq 2$, $A=Circ(-a,b,b,\ldots,b) \in
\mathcal{M}_n(\mathbb{C})$, and $B\in \mathcal{B}(\mathcal{H}).$ Then
\begin{equation}
\max \big\{|a+b|, |(n-1)b-a| \big\}\|B\| \leq \|A\otimes B\|_p \leq \min
\big\{|a+b|+ n|b|, \, |a|+(n-1)|b| \big\} \|B\|, \ \forall p \in [1,\infty].
\end{equation}
\end{theorem}

Before proving the result, we explain the quantity $\kappa$ and why it is
$> 1$. The lower bound in~\eqref{--3} easily follows from the fact that
$\|A\|_p\geq r(A)$ for all $A\in \mathcal{M}_n(\mathbb{C})$ and $p \in
[1, \infty]$. 
The upper bound in~\eqref{--3} was shown in~\cite{Khare} using harmonic
analysis; we now provide a few details. Consider the matrix $K = {\bf
1}_{n\times n}$ as an operator on $\mathcal{P}_{n-1},$ the space of all
polynomials of degree at most $n-1.$ Explicitly, for each
$p(z)=a_0+a_1z+\cdots+a_{n-1}z^{n-1}\in \mathcal{P}_{n-1}$,
$(Kp)(z):=(a_0+a_1+\cdots+a_{n-1})\phi (z),$ where
$\phi(z)=1+z+z^2+\cdots+z^{n-1}.$
With the above, it was shown in \cite[Section 6]{Khare} that 
\[
\|Circ(-a,b,b,\ldots,b)\|_p\leq a+b+nb \| \phi\|_{L^1(\mathbb{T})},
\qquad \forall a,b \in [0,\infty),
\]
where $\kappa = \| \phi\|_{L^1(\mathbb{T})}$ is the $L^1$-norm on the
circle $\mathbb{T}$, with respect to the normalized Haar measure.  

Having defined $\kappa$, its $L^1$-norm is easily estimated:
\[
1 = \bigg|\int_{0}^{2\pi}\phi(e^{i\theta})\frac{d\theta}{2\pi}\bigg| \leq
\int_{0}^{2\pi}|\phi(e^{i\theta})|\frac{d\theta}{2\pi} = \|
\phi\|_{L^1(\mathbb{T})} = \kappa. 
\]
In fact this inequality is strict ($\kappa > 1$), because equality occurs
if and only if $\phi$ is a positive-valued -- in particular,
$\mathbb{R}$-valued -- function on the circle, which is clearly false
(e.g.\ since $\phi$ is a non-constant entire function on the complex
plane).

We now negatively answer a question asked in \cite{Khare}:
\textit{For $a,b \in [0,\infty)$, does the estimation $\|A\|_p\leq a+b+nb
\| \phi\|_{L^1(\mathbb{T})}$ lead to a precise formula for $\|A\|_p$?}
While Theorem~\ref{Tfinal} provides a negative answer via a
strict improvement, we show a strict improvement using even simpler
means. Namely, one can easily find a better estimation directly from the
decomposition $A=-(a+b) I_n + b {\bf 1}_{n\times n}$:
\[
\|A\|_p\leq (a+b) +b \|{\bf 1}_{n\times n}\|_p= a+b + nb < a + b + n b \kappa,
\]
which holds because $\|{\bf 1}_{n\times n}\|_p=n$ (since ${\bf
1}_{n\times n}=Circ(1,1,\ldots,1)$).

Finally, we end this section with

\begin{proof}[Proof of Theorem~\ref{Tfinal}]
Since $A\otimes B=-(a+b) I_n \otimes B+b {\bf 1}_{n \times n}\otimes B$,
we get 
$\|A\otimes B\|_p \leq \|(a+b) I_n \otimes B\|_p+ \|b {\bf 1}_{n \times
n}\otimes B\|_p$. The upper bound now follows by using
Theorem~\ref{thm4-1} (twice):
\[
\|A\otimes B\|_p \leq (|a+b|  + n|b|) \| B\|, \qquad
\|A\otimes B\|_p \leq (|a| + (n-1)|b|) \| B\|.
\]
\blue{To show the lower bound, take a nonzero sequence of vectors $\{
x_m\}_{m=1}^{\infty} \subseteq \mathcal{H}$ such that $\|Bx_m\|\to \|B\|
\|x_m\|.$ Let $z_m=(x_m,x_m,\ldots,x_m)^T, z_m'=(x_m,x_m,0,\ldots,0)^T
\in \mathcal{H}^n$. Then the lower bound follows:
\begin{align*}
\|A\otimes B\|_p \geq &\ \lim_{m\to \infty} \frac{\| (A\otimes B)z_m
\|_p}{\|z_m\|_p} = |(n-1)b-a| \|B\|,\\
\|A\otimes B\|_p \geq &\ \lim_{m\to \infty} \frac{\| (A\otimes B)z_m'
\|_p}{\|z_m'\|_p} = |a+b| \, \|B\|. \qedhere
\end{align*}}
\end{proof}


\section{Estimations for the roots of a polynomial}\label{roots}

As a final application of the numerical radius formula for circulant
matrices, we develop a new estimate for the roots of an arbitrary
complex polynomial. Consider a monic polynomial
\[
p(z)=z^n+a_{n-1}z^{n-1}+\cdots+a_1z+a_0, \qquad n \geq 2,
\]
where $a_0, a_1, \ldots, a_{n-1} \in \mathbb{C}$.
The Frobenius companion matrix $C(p)$ of $p(z)$ is given by
\[
C(p)=\begin{bmatrix}
-a_{n-1} \ \ -a_{n-2}\ \ \ldots \ \ -a_1 & -a_0 \\
I_{n-1} & {\bf 0}_{(n-1) \times 1}
\end{bmatrix}.
\]

It is well known  (see \cite[pp.~316]{Horn2}) that all eigenvalues of
$C(p)$ are exactly the roots of the polynomial $p(z)$. By this argument
and using the numerical radius inequality for $C(p)$, Fujii and Kubo
\cite{Fujii} proved that if $\lambda$ is a root of the polynomial $p(z),$
then
\begin{equation}\label{fuji}
|\lambda| \leq \cos \frac{\pi}{n+1} + \frac12 \left(|a_{n-1}|+
\sqrt{|a_0|^2+|a_1|^2+\cdots+|a_{n-1}|^2} \right).
\end{equation}

Here, using the numerical radius for circulant matrices in
Corollary~\ref{Copnorm} (or \cite{Khare}), we obtain a new estimation
formula for the roots of $p(z)$. We need the following known lemma.

\begin{lemma}[\cite{Fujii}]\label{lemma-rank1}
Let $D=\begin{bmatrix} a_1 \ \ a_2 \ \ \ldots \ \ a_n \\ {\bf 0}_{(n-1)
\times n} \end{bmatrix}$, where $a_1, a_2, \ldots, a_{n}\in \mathbb{C}$.
Then
\[
w(D)=\frac12 \left(|a_{1}|+ \sqrt{|a_1|^2+|a_2|^2+\cdots+|a_n|^2}
\right).
\]
 \end{lemma}

\begin{theorem} \label{est-poly}
If $\lambda$ is a root of $p(z),$ then
	\begin{eqnarray*}  
 |\lambda| &\leq&   1+ \frac12 \left(|a_{n-1}|+ \sqrt{|a_0+1|^2+|a_1|^2+|a_2|^2+\cdots+|a_{n-1}|^2} \right).
	\end{eqnarray*}
\end{theorem}

\begin{proof}
Consider $C(p)=A+D$, where
\begin{eqnarray*}
A=\begin{bmatrix}
{\bf 0}_{1 \times (n-1)} & 1 \\
I_{n-1} & {\bf 0}_{(n-1) \times 1}
\end{bmatrix} \,\,\,\text{and}\,\,\,
D = -\begin{bmatrix}
a_{n-1} \ \ a_{n-2} \ \ \ldots \ \ a_1 \ \ (a_0+1) \\ {\bf 0}_{(n-1) \times n}
\end{bmatrix}.
\end{eqnarray*}
\blue{Since $A=Circ(0,0,\ldots,0,1)$ is a unitary (permutation) matrix,
clearly $w(A)=1$.} 
Also, Lemma~\ref{lemma-rank1} yields
$w(D)=\frac12 \left(|a_{n-1}|+
\sqrt{|a_0+1|^2+|a_1|^2+|a_2|^2+\cdots+|a_{n-1}|^2} \right).$ This yields
the result:
\[
|\lambda| \leq w(C(p))\leq w(A)+w(D)=1+\frac12 \left(|a_{n-1}|+ \sqrt{|a_0+1|^2+|a_1|^2+|a_2|^2+\cdots+|a_{n-1}|^2} \right). \qedhere
\]
\end{proof}

\begin{remark}
Suppose $p(z)=z^n+a_{n-1}z^{n-1}+\cdots+a_1z+a_0,$ where $|a_0+1|< |a_0|$
(e.g.\ where $Re(a_0)\leq -1$). For this class of polynomials,
Theorem~\ref{est-poly} gives a stronger estimate than~\eqref{fuji} for
all sufficiently large $n$.
\end{remark}

\section{Some questions arising from
Theorem~\ref{Tref}}\label{Sspeculation}

\blue{Finally, we consider some questions that arise naturally from
Theorem~\ref{Tref}. We begin with a variant of the Spectral Theorem for
normal matrices. Note that in this line of work, one considers three
non-negative real numbers associated to a complex square matrix $A$:
(a)~the spectral radius $r(A) = \max_{\lambda \in \sigma(A)} |\lambda|$;
(b)~the numerical radius $w(A) = \sup_{\| x \|=1} |\langle Ax,x
\rangle|$; and
(c)~the spectral norm $\| A \| = \sup_{\| x \|=1} \|Ax\|$.
It is well known that $r(A) \leq w(A) \leq \| A\|$, and if $\|A\| = w(A)$
then $\|A\| = r(A)$ too (see Lemma~\ref{lem-Ref}).

Now, Theorem~\ref{Tref} says that if $w(A) = \|A\| = r(A)$, then $A$ has
``partial diagonalizability'': each Jordan block for each maximum-modulus
eigenvalue is $1 \times 1$. Here the reader will recall the Spectral
Theorem, which says that if $A$ is normal (a strictly stronger condition
than $w(A) = \|A\|$) then $A$ is diagonalizable -- and conversely. This
leads to a natural question:

\begin{question}
Suppose $A \in \mathcal{M}_n(\mathbb{C})$. What condition on the
``diagonalizability'' side is equivalent to $w(A) = \|A\| = r(A)$?
Dually, what relations between $r(A), w(A)$, and $\|A\|$ (or more) is
equivalent to each Jordan block of each maximum-modulus eigenvalue being
$1 \times 1$?
\end{question}

The second question has essentially been answered by
Goldberg--Tadmor--Zwas~\cite{GTZ}:

\begin{theorem}[{\cite[Theorem~1]{GTZ}}]
Let $A \in \mathcal{M}_n(\mathbb{C})$ and let its eigenvalues $a_1,
\dots, a_n$ be as in Theorem~\ref{Tref}. Then $w(A) = r(A)$ if and only
if $A$ is unitarily similar to ${\rm diag}(a_1, \dots, a_k) \oplus A'$,
with $A'$ lower triangular and $w(A') \leq r(A)$.
\end{theorem}

The only case where we can fully answer the first question above is for
$n=2$:

\begin{prop}
For any matrix $A \in \mathcal{M}_n(\mathbb{C})$, if $A$ is normal then
$w(A) = \|A\|$. The converse holds if $n=2$.
\end{prop}

\begin{proof}
For the first statement, as both sides are weakly unitarily invariant,
the claim reduces to that for $A$ diagonal, where it is easily verified.
The converse is also \cite[Corollary~2]{GZ}, and its proof is immediate
(and the same in \textit{loc.\ cit.}\ as here, given
Theorem~\ref{Tref}(1)): if $n=2$ then applying Theorem~\ref{Tref}(1), $A$
is unitarily equivalent to a diagonal matrix. Hence $A$ is normal.
\end{proof}

We end with another question that is more directly related to
Theorem~\ref{Tref}.

\begin{question}\label{Qforallm}
Fix an integer $n \geq 2$. Can one characterize all matrices $A \in
\mathcal{M}_n(\mathbb{C})$ such that $w(A) = \|A\|$ and $w(A^{\circ m}) =
w^m(A)$ for all $m \geq 1$?
\end{question}

Here is a partial answer to the question.

\begin{theorem}\label{Tforallm}
Let $\mathcal{S}$ comprise all complex square matrices (of all sizes)
which affirmatively answer Question~\ref{Qforallm}.
\begin{enumerate}
\item Then $\mathcal{S}$ is closed under:
\begin{enumerate}
\item taking block direct sums,
\item rescaling by any $z \in \mathbb{C}$,
\item conjugating by permutation matrices, and
\item taking Kronecker products.
\end{enumerate}

\item $\mathcal{S}$ includes all matrices of the form
$DP \oplus T$,
where $D$ is a unitary diagonal matrix (with diagonal entries in $S^1$),
$P$ is a permutation matrix, and $T$ is any contraction, i.e.\ a matrix
with $\|T\| \leq 1$.

\item If $A \in \mathcal{S}$ has rank one, then the converse is true: $A$
is a diagonal matrix with one nonzero complex entry.
\end{enumerate}
\end{theorem}

It is natural to wonder if Theorem~\ref{Tforallm} provides all solutions
to Question~\ref{Qforallm}.

\begin{proof}\hfill
\begin{enumerate}
\item The first part follows by using that $w(X \oplus Y) = \max( w(X),
w(Y))$. The second and third parts are easily shown. For the fourth, if
$A,B \in \mathcal{S}$, then using Theorem~\ref{cor1}, we get
\[
w(A \otimes B) = w(A) w(B) = \|A\| \, \|B\| = \| A \otimes B \|.
\]
Moreover, since $(A \otimes B)^{\circ m} = A^{\circ m} \otimes B^{\circ
m}$ for all $m,A,B$, we compute:
\[
w((A \otimes B)^{\circ m}) = w(A^{\circ m} \otimes B^{\circ m}) =
w(A^{\circ m}) w(B^{\circ m}) = w^m(A) w^m(B) = w^m(A \otimes B),
\]
where the second equality uses Proposition~\ref{th10}(3) and
Theorem~\ref{cor1}. 

\item Note that $DP$ is unitary, so $r(DP) = \|DP\| = 1$ and hence $w(DP)
= 1 \geq \|T\| \geq w(T)$. Hence $w(DP \oplus T) = 1 = \| DP \oplus T
\|$. Moreover, $D^m P$ is also of the same form as $DP$, so
\[
w((DP)^{\circ m}) = w(D^m P) = 1.
\]

Next, by ``padding test vectors by zeros'', and since $T^{\otimes m}$ is
also a contraction,
\[
w(T^{\circ m}) \leq w(T^{\otimes m}) \leq \| T^{\otimes m} \| \leq 1.
\]
Combining these bounds, we get:
\[
w((DP \oplus T)^{\circ m}) = \max ( w(DP)^{\circ m}, w(T^{\circ m}) ) = 1
= w^m(DP \oplus T).
\]

\item This assertion is the converse of the preceding part, in the sense
that if one considers the closure of matrices $DP \oplus T$ under the
operations in (1a), (1b), (1c), then the matrices of rank one in this
closure are precisely the diagonal matrices with one nonzero entry. For
instance, if $DP \oplus T$ has rank one, then $T$ must be zero and $DP$
is invertible, hence $1 \times 1$.

We now prove the assertion. Suppose $A_{n \times n} \in \mathcal{S}$ has
rank one. Write $A = u v^* = u \otimes v^*$, with $0 \neq u,v \in
\mathbb{C}^n$. Thus, $\|A\| = \| u \| \, \| v \|$, and
\cite[Theorem~1]{Fujii} gives that
\[
w(A) = \frac{\| u \| \, \| v \| + |\langle u,v \rangle|}{2}.
\]

By the Cauchy--Schwarz equality, $w(A) = \|A\|$ means that $u,v$ are
proportional. Thus, write $A = \lambda v v^*$, with $\lambda$ and $v$
nonzero. Then for all $m \geq 1$,
\[
|\lambda|^m \| v^{\circ m} \|^2 =
w(A^{\circ m}) = w^m(A) = \|A\|^m = |\lambda|^m \|v\|^{2m}.
\]

Now we will use this equality not for all $m \geq 1$, but for any single
$m \geq 2$. Canceling $|\lambda|^m$ and letting $v = (v_1, \dots,
v_n)^T$, we have
\[
\sum_{j=1}^n |v_j|^{2m} = \left( \sum_{j=1}^n |v_j|^2 \right)^m.
\]
Letting $a_j := |v_j|^2 \geq 0$, we have $\sum_j a_j^m = ( \sum_j
a_j)^m$. This holds if and only if at most one $a_j$ is nonzero. Thus $v
= v_j {\bf e}_j$ for a unique $1 \leq j \leq n$ and $v_j \neq 0$, and the
proof is complete. \qedhere
\end{enumerate}
\end{proof}

\begin{remark}
For completeness, we mention that matrices of the form $DP$ occur in
multiple other settings. They were used by
Hershkowitz--Neumann--Schneider~\cite{HNS} to classify all positive
semidefinite matrices with entries of modulus $0,1$. But even before
that, a folklore fact asserts that such $n \times n$ matrices $DP$ are
precisely the linear isometries of $\mathbb{C}^n$ equipped with the
$\|\cdot\|_p$ norm, for every $2 \neq p \in [1,\infty]$. This is a
special case of the Banach--Lamperti theorem \cite{Banach,Lamperti}.
\end{remark}}

\blue{
\subsection*{Acknowledgments}
We thank Terence Tao for pointing us to the references \cite{GTZ,GZ}.
We also thank the referee for several useful suggestions that improved
the manuscript, including Theorem~\ref{Tref} (which replaces the previous
version that required assuming $A$ normal) and the alternate proof of
Proposition \ref{P2x2}.
P.B.\ was supported by National Post Doctoral Fellowship
PDF/2022/000325 by SERB (Govt.\ of India) and NBHM Post-Doctoral
Fellowship 0204/16(3)/2024/R\&D-II/6747 by National Board for Higher
Mathematics (Govt.\ of India).
A.K.\ was partially supported by SwarnaJayanti Fellowship grants
SB/SJF/2019-20/14 and DST/SJF/MS/2019/3 from SERB and DST (Govt.\ of
India) and by a Shanti Swarup Bhatnagar Award from CSIR (Govt.\ of
India).

\subsection*{Declaration of competing interest}
The authors have no relevant financial or non-financial interests to disclose.

\subsection*{Data availability}
Data sharing not applicable to this article as no datasets were generated
or analysed during the current study.
}

\bibliographystyle{amsplain}


\end{document}